\documentclass[11pt]{amsart}
\usepackage{amsmath,amssymb,latexsym,esint,cite,mathrsfs}
\usepackage{verbatim,wasysym,cite,relsize,color}
\usepackage[left=2.7cm,right=2.7cm,top=2.7cm,bottom=2.7cm]{geometry}
\usepackage[latin1]{inputenc}
\usepackage{microtype}
\usepackage{color,enumitem,graphicx}
\usepackage[colorlinks=true,urlcolor=blue, citecolor=red,linkcolor=blue,
linktocpage,pdfpagelabels, bookmarksnumbered,bookmarksopen]{hyperref}
\usepackage[hyperpageref]{backref}
\usepackage[english]{babel}
\usepackage{physics} 

\renewcommand{\epsilon}{{\varepsilon}}

\numberwithin{equation}{section}
\newtheorem{theorem}{Theorem}[section]
\newtheorem{lemma}[theorem]{Lemma}

\newtheorem{definition}[theorem]{Definition}
\newtheorem{proposition}[theorem]{Proposition}
\newtheorem{corollary}[theorem]{Corollary}

\title[System NLS quadratic interaction]
{Existence, stability of standing waves and the characterization of finite time blow-up solutions for a system NLS with quadratic interaction}

\author[V. D. Dinh]{Van Duong Dinh}
\address[V. D. Dinh]{Institut de Math\'ematiques de Toulouse UMR5219, Universit\'e Toulouse CNRS, 31062 Toulouse Cedex 9, France 
and 
Department of Mathematics, HCMC University of Pedagogy, 280 An Duong Vuong, Ho Chi Minh, Vietnam}
\email{dinhvan.duong@math.univ-toulouse.fr}


\begin{document}
	
	\begin{abstract}
	We study the existence and stability of standing waves for a system of nonlinear Schr\"odinger equations with quadratic interaction in dimensions $d\leq 3$. We also study the characterization of finite time blow-up solutions with minimal mass to the system under mass resonance condition in dimension $d=4$. Finite time blow-up solutions with minimal mass are showed to be (up to symmetries) pseudo-conformal transformations of a ground state standing wave.
	\end{abstract}

	\maketitle

	\section{Introduction}
	\label{S:0}
	We study the system NLS equations
	\begin{equation} \label{Syst}
		\left\{ 
		\renewcommand*{\arraystretch}{1.5}
		\begin{array}{rcl}
			i\partial_t u + \frac{1}{2m} \Delta u & = & \lambda v \overline{u}, \\
			i\partial_t v + \frac{1}{2M} \Delta v & = & \mu u^2,
		\end{array}
		\right.
	\end{equation}
	where $u$ and $v$ are complex-valued functions of $(t,x) \in \mathbb{R} \times \mathbb{R}^d$, $\Delta$ is the Laplacian in $\mathbb{R}^d$, $m$ and $M$ are positive constants, $\lambda$ and $\mu$ are complex constants, and $\overline{u}$ is the complex conjugate of $u$. 
	
	The system \eqref{Syst} is related to the Raman amplification in a plasma (see e.g. \cite{CCO}). The system \eqref{Syst} is also regarded as a non-relativistic limit of the system of nonlinear Klein-Gordon equations
	\[
	\left\{ 
	\renewcommand*{\arraystretch}{1.5}
	\begin{array}{rcl}
	\frac{1}{2c^2m}\partial^2_t u - \frac{1}{2m} \Delta u + \frac{mc^2}{2} u& = & -\lambda v \overline{u}, \\
	\frac{1}{2c^2M}\partial^2_t v - \frac{1}{2M} \Delta v + \frac{Mc^2}{2} v& = & -\mu u^2,
	\end{array}
	\right.
	\]
	under the mass resonance condition
	\begin{align} \label{mas-res}
		M=2m.
	\end{align}
	Indeed, the modulated wave functions $(u_c,v_c):= (e^{itmc^2} u, e^{itMc^2} v)$ satisfy
	\begin{align}\label{klei-gord}
		\left\{ 
		\renewcommand*{\arraystretch}{1.5}
		\begin{array}{rcl}
			\frac{1}{2c^2m} \partial^2_t u_c - i\partial_t u_c - \frac{1}{2m} \Delta u_c &=& - e^{itc^2(2m-M)} \lambda v_c \overline{u}_c,\\
			\frac{1}{2c^2M} \partial^2_t v_c - i\partial_t v_c - \frac{1}{2M} \Delta v_c &=& - e^{itc^2(M-2m)} \mu u^2_c.
		\end{array}
		\right.
	\end{align}
	We see that the phase oscillations on the right hand sides vanish if and only if \eqref{mas-res} holds, and the system \eqref{klei-gord} formally yields \eqref{Syst} as the speed of light $c$ tends to infinity.
	
	By the standard scaling argument, the system \eqref{Syst} has the critical function space $H^{\frac{d}{2}-2}$, where $H^s$ is the usual Sobolev space of order $s$. In particular, the system \eqref{Syst} is $L^2$-critical for $d=4$, is $H^1$-critical for $d=6$ and is intercritical for $d=5$. 
	
	In \cite{HOT}, Hayashi-Ozawa-Takana studied the Cauchy problem for \eqref{Syst} in $L^2$, $H^1$ and in the weighted $L^2$ space $\langle x \rangle^{-1} L^2 = \mathcal{F}(H^1)$ under mass resonance condition, where $\langle x \rangle = \sqrt{1+ |x|^2}$ is the the Japanese bracket and $\mathcal{F}$ is the Fourier transform. They showed the existence of ground states for \eqref{Syst} by using the variational methods. They also pointed out explicit finite time blow-up solutions for \eqref{Syst} under the mass resonance condition in dimension $d=4$. Recently, Hamano in \cite{Hamano} showed the sharp threshold for scattering and blow-up for \eqref{Syst} under the mass resonance condition in dimension $d=5$. 
	
	Let us recall the local well-posedness in $H^1$ for \eqref{Syst} due to \cite{HOT}. The Cauchy problem \eqref{Syst} with data $(u(t_0), v(t_0)) = (u_0,v_0)$ given at $t=t_0$ is treated in the form of system of integral equations
	\begin{equation} \label{int-Syst}
		\left\{
		\renewcommand*{\arraystretch}{2}
		\begin{array}{rcl}
			u(t) &=& U_m(t-t_0) u_0 - i \mathlarger{\int}_{t_0}^t U_m(t-\tau) \lambda v(\tau) \overline{u}(\tau) d\tau, \\
			v(t) &=& U_M(t-t_0) v_0 - i \mathlarger{\int}_{t_0}^t U_M(t-\tau) \mu u^2(\tau) d\tau,
		\end{array}
		\right.
	\end{equation}
	where $U_m(t) = \exp \left(i\frac{t}{2m}\Delta\right)$ and $U_M(t) = \exp \left(i \frac{t}{2M}\Delta \right)$ are free propagators with masses $m$ and $M$ respectively. 
	
	To ensure the conservation law of total charge, it is natural to consider the following condition:
	\begin{align} \label{mas-con}
		\exists ~ c \in \mathbb{R} \backslash \{0\} \ : \ \lambda = c \overline{\mu}. 
	\end{align} 
	\begin{proposition}[LWP in $H^1$ \cite{HOT}] \label{prop-lwp-h1}
		Let $d\leq 6$ and let $\lambda$ and $\mu$ satisfy \eqref{mas-con}. Then for any $(u_0,v_0) \in H^1\times H^1$, there exists a unique paire of local solutions $(u,v) \in Y(I)\times Y(I)$ of \eqref{int-Syst}, where
		\begin{align*}
			Y(I) = (C\cap L^\infty)(I,H^1) \cap L^4(I,W^{1,\infty}) &\text{ for } d=1, \\
			Y(I) = (C\cap L^\infty)(I,H^1) \cap L^{q_0}(I,W^{1,{r_0}}) &\text{ for } d=2, 
		\end{align*}
		where $0<\frac{2}{q_0}=1-\frac{2}{r_0}<1$ with $r_0$ sufficiently large, 
		\begin{align*}
			Y(I) = (C\cap L^\infty)(I, H^1) \cap L^2(I, W^{1,\frac{2d}{d-2}}) &\text{ for } d\geq 3.
		\end{align*}
		Moreover, the solution satisfies the mass and energy conservation laws: for all $t\in I$,
		\begin{align*}
		M(u(t),v(t)) &:= \|u(t)\|^2_{L^2} + c \|v(t)\|^2_{L^2} = M(u_0,v_0), \\
		E(u(t),v(t)) &:= \frac{1}{2m}\|\nabla u(t)\|^2_{L^2} + \frac{c}{4M} \|\nabla v(t)\|^2_{L^2} + \emph{Re} (\lambda \langle v(t), u^2(t) \rangle ) = E(u_0,v_0),
		\end{align*}
		where $\langle \cdot, \cdot \rangle$ is the scalar product in $L^2$. 
	\end{proposition}
	Note that for $d\leq 5$, the existence time depends only on $\|u_0\|_{H^1} + \|v_0\|_{H^1}$. However, for $d=6$, the existence time depends not only on $H^1$-norm of initial data but also on the profile of the initial data.
	
	From now on, we assume that $\lambda$ and $\mu$ satisfy \eqref{mas-con} with $c>0$ and $\lambda \ne 0, \mu \ne 0$. By change of variables
	\[
	u(t,x) \mapsto \sqrt{\frac{c}{2}} |\mu| u \left(t,\sqrt{\frac{1}{2m}} x \right), \quad v(t,x) \mapsto -\frac{\lambda}{2} v\left( t, \sqrt{\frac{1}{2m}} x\right),
	\]
	the system \eqref{Syst} becomes 
	\begin{equation} \label{wor-Syst}
		\left\{ 
		\renewcommand*{\arraystretch}{1.3}
		\begin{array}{rcl}
			i\partial_t u + \Delta u & = & -2 v \overline{u}, \\
			i\partial_t v + \kappa \Delta v & = & - u^2,
		\end{array}
		\right.
	\end{equation}
	where $\kappa=\frac{m}{M}$ is the mass ratio. In the sequel, we only consider the system \eqref{wor-Syst} with initial data $(u(0), v(0)) = (u_0,v_0)$. Note that the mass now becomes
	\[
	M(u(t),v(t)) = \|u(t)\|^2_{L^2} + 2 \|v(t)\|^2_{L^2},
	\]
	and the energy is
	\[
	E(u(t),v(t)) = \frac{1}{2} (\|\nabla u(t)\|^2_{L^2} + \kappa \|\nabla v(t)\|^2_{L^2} ) - \text{Re}( \langle v(t), u^2(t)\rangle).
	\]
	The local well-posedness in $H^1$ for \eqref{wor-Syst} reads as follows.
	\begin{proposition} [LWP in $H^1$] \label{prop-lwp-wor}
		Let $d\leq 6$. Then for any $(u_0, v_0) \in H^1 \times H^1$, there exists a unique pair of local solutions $(u,v) \in Y(I) \times Y(I)$ of \eqref{wor-Syst}. Moreover, the solution satisfies the conservation of mass and energy: for all $t \in I$,
		\begin{align*}
		M(u(t),v(t)) &:= \|u(t)\|^2_{L^2} + 2 \|v(t)\|^2_{L^2} = M(u_0,v_0), \\
		E(u(t),v(t)) &:= \frac{1}{2} (\|\nabla u(t)\|^2_{L^2} + \kappa \|\nabla v(t)\|^2_{L^2}) - \emph{Re} (\langle v(t),u^2(t)\rangle) = E(u_0,v_0).
		\end{align*}
	\end{proposition}
	In the first part of the sequel, we study the existence and stability of ground state standing wave solutions for \eqref{wor-Syst}. A standing wave for \eqref{wor-Syst} is a solution of the form $(u(t),v(t))$ with $u(t,x) = e^{i\omega_1 t} \phi(x)$ and $v(t,x) = e^{i\omega_2 t} \psi(x)$, where $\omega_1$ and $\omega_2$ are real numbers and $\phi, \psi: \mathbb{R}^d \rightarrow \mathbb{C}$ satisfy the coupled system of elliptic equations
	\begin{align} \label{sys-ell-equ}
	\left\{ 
	\renewcommand*{\arraystretch}{1.3}
	\begin{array}{rcl}
	-\Delta \phi + \omega_1 \phi & = & 2 \psi \overline{\phi}, \\
	-\kappa \Delta \psi + \omega_2 \psi & = & \phi^2,
	\end{array}
	\right.
	\end{align}
	In \cite{HOT}, Hayashi-Ozawa-Tanaka proved the existence of real-valued ground states for \eqref{sys-ell-equ} with $\omega_2=2\omega_1=2\omega>0$ and $ d\leq 5$. We recall that a pair of complex-valued functions $(\phi,\psi) \in H^1 \times H^1$ is called a ground state for \eqref{sys-ell-equ} with $\omega_2=2\omega_1=2\omega$ if it minimizes the associated functional
	\begin{align*}
	S_\omega(u,v) :&= E(u,v) +\frac{\omega}{2} M(u,v) \\
	&=\frac{1}{2} (\|\nabla u\|^2_{L^2} + \kappa \|\nabla v\|^2_{L^2}) + \frac{\omega}{2} (\|u\|^2_{L^2} + 2 \|v\|^2_{L^2}) - \text{Re}(\langle v,u^2\rangle)
	\end{align*}
	among all non-zero solutions of \eqref{sys-ell-equ}. The proof is based on the variational argument using Strauss's compactness embedding $H^1_{\text{rad}}(\mathbb{R}^d) \subset L^3(\mathbb{R}^d)$ for $2 \leq d \leq 5$. The case $d=1$, they employed a concentration-compactness argument using Palais-Smale sequence. In this paper, our approach is different and is based on the concentration-compactness method of Lions \cite{Lions1}. Given any $a,b>0$, we look for solutions $(\phi,\psi) \in H^1 \times H^1$ of \eqref{sys-ell-equ} satisfying $\|\phi\|^2_{L^2} = a$ and $\|\psi\|^2_{L^2}=b$. Such solutions are of interest in physics (often referred to as normalized solutions). To this end, we consider, for $d\leq 3$ and $a,b>0$, the following variational problem
	\begin{align} \label{prob-1}
	I(a,b):= \inf \{ E(u,v) \ : \ (u,v) \in H^1 \times H^1, \|u\|^2_{L^2} =a, \|v\|^2=b \}.
	\end{align}
	We denote the set of nontrivial minimizers for $I(a,b)$ by 
	\[
	\mathcal{G}_{a,b} = \{ (u,v) \in H^1 \times H^1 \ : \ E(u,v) = I(a,b), \|u\|^2_{L^2} =a, \|v\|^2_{L^2}=b \}.
	\]
	Our first result is the existence of minimizers for \eqref{prob-1}. 
	\begin{theorem} \label{theo-exi-min}
		Let $d\leq 3$ and $a,b>0$. Then the following properties hold:
		\begin{itemize}
			\item[(1)] The set $\mathcal{G}_{a,b}$ is not empty. Any minimizing sequence $(u_n,v_n)_{n\geq 1}$ for $I(a,b)$ is relatively compact in $H^1 \times H^1$ up to translations. That is, there exist $(y_n)_{n\geq 1} \subset \mathbb{R}^d$ and $(u,v) \in H^1\times H^1$ such that $(u_n(\cdot+y_n), v_n(\cdot+y_n))_{n\geq 1}$ has subsequence converging strongly to $(u,v)$ in $H^1 \times H^1$. Moreover, $(u,v) \in \mathcal{G}_{a,b}$.
			\item[(2)] 
			\begin{align}\label{lit-gro-1}
			\inf_{(w,z) \in G_{a,b}, y \in \mathbb{R}^d} \|(u_n(\cdot+y), v_n(\cdot+y)) - (w,z)\|_{H^1\times H^1} \rightarrow 0 \text{ as } n\rightarrow \infty.
			\end{align}
			\item[(3)] 
			\begin{align} \label{lit-gro-2}
			\inf_{(w,z) \in \mathcal{G}_{a,b}} \|(u_n,v_n) - (w,z)\|_{H^1\times H^1} \rightarrow 0 \text{ as } n\rightarrow \infty.
			\end{align}
			\item[(4)] Each $(u,v) \in \mathcal{G}_{a,b}$ is a classical solution of \eqref{sys-ell-equ} for some $\omega_1, \omega_2 \in \mathbb{R}$. Moreover, there exist $\theta_1, \theta_2 \in \mathbb{R}$ and nonnegative functions $\vartheta, \zeta$ such that $u(x) = e^{i \theta_1} \vartheta(x)$ and $v(x) = e^{i\theta_2} \zeta(x)$ for all $x \in \mathbb{R}^d$.
		\end{itemize}
	\end{theorem}
	Note that we will use in this paper the following convention: a minimizing sequence for $I(a,b)$ is defined as a sequence $(u_n,v_n)_{n\geq 1} \subset H^1 \times H^1$ such that $\|u_n\|^2_{L^2} \rightarrow a$, $\|v_n\|^2_{L^2} \rightarrow b$ and $E(u_n, v_n) \rightarrow I(a,b)$ as $n\rightarrow \infty$.
	
	The proof of Theorem $\ref{theo-exi-min}$ is based on the concentration-compactness method of Lions \cite{Lions1}. Similar arguments have been used in \cite{Bhattarai-frac} (see also \cite{AB, Ardila, Bhattarai-syst}) to show the existence and stability of standing waves.
	
	At the moment, we do not know the set $\mathcal{G}_{a,b}$ are orbital stable under the flow of \eqref{wor-Syst} or not. The main issue is that the quantities $\|u(t)\|^2_{L^2}$ and $\|v(t)\|^2_{L^2}$ are not conserved under the flow of \eqref{wor-Syst}. We only have the conservation of mass $M(u(t),v(t)) = \|u(t)\|^2_{L^2} + 2\|v(t)\|^2_{L^2} = M(u_0,v_0)$ and $E(u(t), v(t))=E(u_0,v_0)$. The orbital stability of standing waves for \eqref{wor-Syst} is thus related to the following variational problem
	\begin{align} \label{prob-2}
	J(c):= \inf \{ E(u,v) \ : \ (u,v) \in H^1 \times H^1, M(u,v) =c \}.
	\end{align}
	The set of nontrivial minimizers for $J(c)$ is denoted by
	\[
	\mathcal{M}_c = \inf\{ (u,v) \in H^1 \times H^1 \ : \ E(u,v) = J(c), M(u,v)=c \}.
	\]
	Our next result is the existence of minimizers for $J(c)$.
	\begin{theorem} \label{theo-exi-min-J}
		Let $d\leq 3$ and $c>0$. Then the following properties hold:
		\begin{itemize}
			\item[(1)] The set $\mathcal{M}_c$ is not empty. Any minimizing sequence $(u_n,v_n)_{n\geq 1}$ for $J(c)$ is relatively compact in $H^1\times H^1$ up to translations. That is, there exist $(y_n)_{n\geq 1} \subset \mathbb{R}^d$ and $(u,v)\in H^1 \times H^1$ such that $(u_n(\cdot+y_n),v_n(\cdot+y_n))_{n\geq 1}$ has subsequence converging strongly to $(u,v)$ in $H^1\times H^1$. Moreover, $(u,v) \in \mathcal{M}_c$.
			\item[(2)]
			\begin{align} \label{lit-gro-J-1}
			\inf_{(w,z) \in \mathcal{M}_c, y \in \mathbb{R}^d} \|(u_n(\cdot+y),v_n(\cdot+y)) -(w,z)\|_{H^1\times H^1} \rightarrow 0 \text{ as } n\rightarrow \infty.
			\end{align}
			\item[(3)] 
			\begin{align} \label{lit-gro-J-2}
			\inf_{(w,z) \in \mathcal{M}_c} \|(u_n,v_n) -(w,z)\|_{H^1\times H^1} \rightarrow 0 \text{ as } n\rightarrow \infty.
			\end{align}
			\item[(4)] Each $(u,v) \in \mathcal{M}_c$ is a classical solution of \eqref{sys-ell-equ} for some $\omega_2=2\omega_1=2\omega >0$. Moreover, there exist $\theta_1, \theta_2 \in \mathbb{R}$ and nonnegative functions $\vartheta,\zeta$ such that $u(x) = e^{i\theta_1} \vartheta(x)$ and $v(x)=e^{i\theta_2} \zeta(x)$ for all $x \in \mathbb{R}^d$.
			\item[(5)] Each $(u,v)\in \mathcal{M}_c$ is a ground state for \eqref{sys-ell-equ} with $\omega_2=2\omega_1=2\omega>0$.
		\end{itemize}
	\end{theorem}
	We also have the orbital stability of standing waves for \eqref{wor-Syst}. 
	\begin{theorem} \label{theo-sta}
		Let $d\leq 3$ and $c>0$. Then the set $\mathcal{M}_c$ is stable under the flow of \eqref{wor-Syst} in the sense that for any $\epsilon>0$, there exists $\delta>0$ such that if $(u_0,v_0) \in H^1 \times H^1$ satisfies
		\[
		\inf_{(w,z) \in \mathcal{M}_c} \|(u_0,v_0) - (w,z) \|_{H^1\times H^1} <\delta,
		\]
		then the global solution $(u(t),v(t))$ of \eqref{wor-Syst} with initial data $(u_0,v_0)$ satisfies
		\[
		\sup_{t\geq 0} \inf_{(w,z) \in \mathcal{M}_c} \|(u(t),v(t))-(w,z)\|_{H^1\times H^1} < \epsilon.
		\]
	\end{theorem}
	
	The second part of this paper is devoted to the existence of blow-up solutions and the characterization of finite time blow-up solutions with minimal mass for \eqref{wor-Syst} under the mass resonance condition in dimension $d=4$. Let $d=4$ and $\kappa=\frac{1}{2}$. We consider the following system NLS:
	\begin{equation} \label{mas-res-Syst}
	\left\{ 
	\renewcommand*{\arraystretch}{1.5}
	\begin{array}{rcl}
	i\partial_t u + \Delta u & = & -2 v \overline{u}, \\
	i\partial_t v + \frac{1}{2} \Delta v & = & - u^2.
	\end{array}
	\right.
	\end{equation}
	We first recall the sharp Gagliardo-Nirenberg inequality related to \eqref{mas-res-Syst}, namely
	\[
	P(u,v) \leq C_{\text{opt}} K(u,v) \sqrt{M(u,v)},
	\]
	where 
	\[
	P(u,v) = \text{Re} (\langle v, u^2 \rangle), \quad K(u,v)= \|\nabla u\|^2_{L^2}+ \frac{1}{2} \|\nabla v\|^2_{L^2}, \quad M(u,v)= \|u\|^2_{L^2} + 2 \|v\|^2_{L^2}.
	\]
	We have
	\[
	\frac{1}{C_{\text{opt}}} := \inf \{ J(u,v) \ : \ (u,v) \in \mathcal{P} \},
	\]
	where
	\[
	J(u,v) = \frac{K(u,v) \sqrt{M(u,v)}}{ P(u,v)},
	\]
	and 
	\[
	\mathcal{P} := \{ (u,v) \in H^1 \times H^1 \backslash \{(0,0)\} \ : \ P(u,v)>0 \}.
	\]
	\begin{theorem} [Sharp Gagliardo-Nirenberg inequality \cite{HOT}] \label{theo-sha-GN-cla}
		Let $d=4$ and $\kappa=\frac{1}{2}$. Then the sharp constant $C_{\emph{opt}}$ is attained by a pair of functions $(\phi_0, \psi_0) \in H^1\times H^1$ which is a positive radially symmetric solution of 
		\begin{equation} \label{ell-equ-mas-res}
		\left\{ 
		\renewcommand*{\arraystretch}{1.3}
		\begin{array}{rcl} 
		-\Delta \phi + \phi & = & 2 \psi \overline{\phi}, \\ 
		-\frac{1}{2} \Delta \psi + 2 \psi & = & \phi^2.
		\end{array}
		\right.
		\end{equation}
		Moreover, 
		\begin{align} \label{sha-con-GN-cla}
		C_{\emph{opt}} = \frac{1}{2\sqrt{M(\phi_0,\psi_0)}}.
		\end{align}
	\end{theorem}
	The main reason for considering the blow-up in dimension $d=4$ comes from the following fact.
	\begin{proposition} [GWP in $H^1$] \label{prop-gwp-wor}
		If $d\leq 3$, then for any $(u_0,v_0) \in H^1 \times H^1$, \eqref{wor-Syst} has a unique pair of solutions $(u,v) \in Y(\mathbb{R}) \times Y(\mathbb{R})$. If $d=4$, then for any $(u_0,v_0) \in H^1 \times H^1$ with 
		\[
		M(u_0,v_0) < M(\phi_0,\psi_0),
		\]
		where $(\phi_0,\psi_0)$ is as in Theorem $\ref{theo-sha-GN-cla}$, \eqref{wor-Syst} has a unique pair of solutions $(u,v) \in Y(\mathbb{R}) \times Y(\mathbb{R})$.
	\end{proposition}
	To see Proposition $\ref{prop-gwp-wor}$, we consider two cases: $d\leq 3$ and $d=4$. In the case $d\leq 3$, we use the Gagliardo-Nirenberg inequality
	\begin{equation}\label{GN}
	\|u\|_{L^3} \leq C \|\nabla u\|_{L^2}^{\frac{d}{6}} \|u\|_{L^2}^{1-\frac{d}{6}},
	\end{equation}
	and the conservation of mass to have
	\begin{align} \label{GN-app}
	|P(u,v)| &\leq c \|v\|_{L^3} \|u\|^2_{L^3} \leq c C^3 \|\nabla u\|_{L^2}^{\frac{d}{3}} \|u\|^{2-\frac{d}{3}}_{L^2} \|\nabla v\|^{\frac{d}{6}}_{L^2} \|v\|^{1-\frac{d}{6}}_{L^2} \nonumber \\
	&\leq c C^3 [K(u,v)]^{\frac{d}{6}} [M(u,v)]^{1-\frac{d}{6}} \left(\frac{1}{\kappa} K(u,v) \right)^{\frac{d}{12}} \left(\frac{1}{2} M(u,v)\right)^{\frac{1}{2}-\frac{d}{12}} \nonumber\\
	&\leq c C^3 \kappa^{-\frac{d}{12}} 2^{\frac{d}{12}-\frac{1}{2}} [M(u,v)]^{\frac{3}{2}-\frac{d}{4}} [K(u,v)]^{\frac{d}{4}} \\
	&= c C^3 \kappa^{-\frac{d}{12}} 2^{\frac{d}{12}-\frac{1}{2}} [M(u_0,v_0)]^{\frac{3}{2}-\frac{d}{4}} [K(u,v)]^{\frac{d}{4}}. \nonumber
	\end{align}
	This implies that $|\text{Re} (\langle v, u^2\rangle)| \leq A [K(u,v)]^{\frac{d}{4}}$ for some constant $A>0$. By the Young inequality, we estimate for any $\epsilon>0$,
	\[
	A[K(u,v)]^{\frac{d}{4}} \leq \epsilon K(u,v) + C(\epsilon, A). 
	\]
	By the conservation of energy, we get
	\[
	E(u,v)= E(u_0,v_0) = \frac{1}{2} K(u,v) - P(u,v) \geq  \left(\frac{1}{2} - \epsilon\right) K(u,v) - C(\epsilon, A).
	\]
	Choosing $\epsilon>0$ small enough, we get the uniform bound on $K(u,v)$. In the case $d=4$, it follows from the sharp Gagliardo-Nirenberg in 4D given in Theorem $\ref{theo-sha-GN-cla}$ that
	\[
	|P(u,v)| \leq \frac{1}{2 \sqrt{M(\phi_0,\psi_0)}} K(u,v) \sqrt{M(u,v)} = \frac{\sqrt{ M(u_0,v_0)}}{2 \sqrt{M(\phi_0,\psi_0)}} K(u,v),
	\]
	where $(\phi_0,\psi_0)$ is as in Theorem $\ref{theo-sha-GN-cla}$. Thus,
	\[
	E(u,v)= E(u_0,v_0) = \frac{1}{2} K(u,v) - P(u,v) \geq \left(\frac{1}{2} - \frac{1}{2} \sqrt{\frac{M(u_0,v_0)}{M(\phi_0,\psi_0)} } \right) K(u,v).
	\]
	This implies that if $M(u_0,v_0) < M(\phi_0,\psi_0)$, then the kinetic energy is bounded uniformly. Therefore, in both cases, the $H^1$-norm of initial data is bounded uniformly in time, and the result follows.
	
	Proposition $\ref{prop-gwp-wor}$ tells us that the dimension $d=4$ is the smallest dimension for which the large data blow-up solutions of \eqref{wor-Syst} may occur. Our next result is the blow-up criteria for \eqref{mas-res-Syst} with non-radial initial data. 
	\begin{theorem} \label{theo-blo-non-rad}
		Let $d=4$ and $\kappa=\frac{1}{2}$. Let $(u_0,v_0) \in H^1 \times H^1$ be such that the corresponding solution (not necessary radial) to \eqref{mas-res-Syst} exists on the maximal time interval $[0,T)$. If $E(u_0,v_0)<0$, then either the solution blows up in finite time or the solution blows up infinite time and there exists a time sequence $(t_n)_{n\geq 1}$ such that $t_n \rightarrow +\infty$ and 
		\begin{align} \label{blo-non-rad}
		\lim_{n\rightarrow \infty} \|(u(t_n), v(t_n))\|_{H^1 \times H^1} = \infty.
		\end{align}
	\end{theorem}
	The proof of this result is based on the argument of Du-Wu-Zhang \cite{DWZ} using the localized virial estimates. One can rule out the infinite time blow-up given in Theorem $\ref{theo-blo-non-rad}$ by considering radially symmetric initial data with negative energy. Note that in the case initial data has finite variance and negative energy, the existence of finite time blow-up solutions was proved by Hayashi-Ozawa-Takana \cite{HOT}.
	\begin{theorem} \label{theo-blo-rad}
		Let $d=4$ and $\kappa=\frac{1}{2}$. Let $(u_0,v_0) \in H^1 \times H^1$ be radial and satisfies $E(u_0,v_0)<0$. Then the corresponding solution to \eqref{mas-res-Syst} blows up in finite time.
	\end{theorem}
	The proof of this result is based on localized radial virial estimates and an argument of \cite{OT}. A similar argument has been used in \cite{Dinh} to show the existence of finite time blow-up solutions for the inhomogeneous nonlinear Schr\"odinger equation.
	\begin{definition} \label{defi-G}
		We denote the set of minimizers of $J$ which are positive radially symmetric solutions of \eqref{ell-equ-mas-res} by $\mathcal{G}$. It follows from \eqref{sha-con-GN-cla} that all elements of $\mathcal{G}$ have the same mass, that is, there exists $M_{\emph{gs}}>0$ such that $M(\phi,\psi)= M_{\emph{gs}}$ for any $(\phi,\psi) \in \mathcal{G}$.
	\end{definition}
	An immediate consequence of Proposition $\ref{prop-gwp-wor}$ and Definition $\ref{defi-G}$ is the following global well-posedness result.
	\begin{corollary} \label{coro-glo-wel}
		Let $d=4$ and $\kappa=\frac{1}{2}$. Let $(u_0,v_0) \in H^1 \times H^1$ be such that
		\[
		M(u_0,v_0) < M_{\emph{gs}}.
		\]
		Then the corresponding solution to \eqref{mas-res-Syst} exists globally in time.
	\end{corollary}
	Corollary $\ref{coro-glo-wel}$ infers that $M_{\text{gs}}$ is the smallest mass for which the finite time blow-up solutions could appear. In \cite{HOT}, Hayashi-Ozawa-Takana constructed explicit solutions which blow up at finite time and have minimal mass $M_{\text{gs}}$. More precisely, they proved the following result (see \cite[Theorem 6.3]{HOT}).
	\begin{proposition} \label{prop-exp-blo-sol}
		Let $d=4$ and $\kappa = \frac{1}{2}$. Let $(\phi,\psi) \in \mathcal{G}$. For $T>0$, let 
		\begin{equation} \label{pse-con-tra}
		\begin{aligned}
		u(t,x) &= \frac{1}{(T-t)^2} \exp \left( -i\frac{|x|^2}{4(T-t)} + i\frac{t}{T(T-t)}\right) \phi\left(\frac{x}{T-t}\right), \\
		v(t,x) &= \frac{1}{(T-t)^2} \exp \left( -i\frac{|x|^2}{2(T-t)} + i\frac{2t}{T(T-t)}\right) \psi\left(\frac{x}{T-t}\right).
		\end{aligned}
		\end{equation}
		Then $(u,v)$ is a solution of \eqref{wor-Syst}, and satisfies:
		\begin{itemize}
			\item[(1)] $u,v \in C^\infty(-\infty,T), H^\infty(\mathbb{R}^4))$, where $H^\infty = \cap_{m\geq 1} H^m$. 
			\item[(2)] $(u(0),v(0)) = \left( \frac{1}{T^2} \exp \left(-i\frac{|x|^2}{4T} \right) \phi\left(\frac{x}{T}\right), \frac{1}{T^2} \exp\left(-i\frac{|x|^2}{2T}\right) \psi\left(\frac{x}{T}\right) \right)$. 
			\item[(3)] $M(u(0),v(0)) = M(\phi,\psi) = M_{\emph{gs}}= \frac{C_{\emph{opt}}^2}{4}$. 
			\item[(4)] $K(u(t),v(t)) = O((T-t)^{-2})$ as $t\uparrow T$. 
			\item[(5)] $P(u(t), v(t)) = O((T-t)^{-2})$ as $t\uparrow T$.
			\item[(6)] $M(u(t),v(t)) \rightarrow M(\phi,\psi) \delta$ weakly star in $\mathcal{D}'(\mathbb{R}^4)$ as $t\uparrow T$, where $\delta$ is the Dirac delta at the origin.
		\end{itemize}
	\end{proposition}
	
	Our next result is the classification of finite time blow-up solutions with minimal mass for \eqref{mas-res-Syst} in the case $d=4$ and $\kappa=\frac{1}{2}$.
	\begin{theorem}[Classification minimal mass blow-up solutions] \label{theo-cla-min-mas}
		Let $d=4$ and $\kappa=\frac{1}{2}$. Let $(u_0,v_0) \in H^1 \times H^1$ be such that $M(u_0,v_0) =M_{\emph{gs}}$. Assume that the corresponding solution to \eqref{mas-res-Syst} blows up in finite time $0<T<+\infty$. Then there exist $(\phi,\psi) \in \mathcal{G}$, $\theta_1, \theta_2 \in \mathbb{R}$ and $\rho>0$ such that
		\[
		u_0(x) = e^{i\theta_1} e^{i\frac{\rho^2}{T}} e^{-i\frac{|x|^2}{4T}} \left(\frac{\rho}{T}\right)^2 \phi\left( \frac{\rho x}{T} \right), \quad v_0(x) = e^{i\theta_2} e^{i\frac{\rho^2}{T}} e^{-i\frac{|x|^2}{2T}} \left(\frac{\rho}{T}\right)^2 \psi\left( \frac{\rho x}{T} \right).
		\]
		In particular, 
		\begin{equation} \label{sol-cla}
		\begin{aligned}
		u(t,x) &= e^{i\theta_1} e^{i\frac{\rho^2}{T-t}} e^{-i\frac{|x|^2}{4(T-t)}} \left( \frac{\rho}{T-t}\right)^2 \phi\left( \frac{\rho x}{T-t} \right), \\
		v(t,x) &= e^{i\theta_2} e^{i\frac{\rho^2}{T-t}} e^{-i\frac{|x|^2}{2(T-t)}} \left( \frac{\rho}{T-t}\right)^2 \psi\left( \frac{\rho x}{T-t} \right).
		\end{aligned}
		\end{equation}
	\end{theorem}
	Note that if we take $\theta_1 =-\frac{1}{T}$, $\theta_2 = -\frac{2}{T}$ and $\rho=1$, then \eqref{sol-cla} becomes \eqref{pse-con-tra}.
	
	The proof of Theorem $\ref{theo-cla-min-mas}$ is inspired by an argument of \cite{HK} using the profile decomposition (see also \cite{CG} for a different approach).
	
	This paper is organized as follows. In Section $\ref{S:1}$, we study the existence and stability of standing waves. More precisely, we study the variational problem $I(a,b)$ in Subsection $\ref{Sub:1.1}$ where Theorem $\ref{theo-exi-min}$ is proved. The variational problem $J(c)$ is studied in Subsection $\ref{Sub:1.2}$, and Theorem $\ref{theo-exi-min-J}$ and Theorem $\ref{theo-sta}$ are proved in this subsection. Section $\ref{S:2}$ is devoted to the existence of blow-up solutions and the characterization of finite time blow-up solutions with minimal mass. The existence of blow-up solutions given in Theorem $\ref{theo-blo-non-rad}$ and Theorem $\ref{theo-blo-rad}$ are considered in Subsection $\ref{Sub:2.1}$. The characterization of finite time blow-up solutions with minimal mass given in Theorem $\ref{theo-cla-min-mas}$ is showed in Subsection $\ref{Sub:2.2}$. 
	
	\section{Existence and stability of standing waves}
	\label{S:1}
	Throughout this section, we use the following notations:
	\[
	K(u,v) := \|\nabla u\|^2_{L^2} + \kappa \|\nabla v\|^2_{L^2}, \quad M(u,v) = \|u\|^2_{L^2} + 2 \|v\|^2_{L^2}, \quad P(u,v) = \text{Re} (\langle v, u^2 \rangle ).
	\]
	\subsection{The variational problem $I(a,b)$}
	\label{Sub:1.1}
	Let $d \leq 3$ and $a,b>0$. We consider the variational problem \eqref{prob-1}.
	\begin{lemma} \label{lem-neg}
		Let $d\leq 3$ and $a,b>0$. Then $-\infty<I(a,b) <0$.
	\end{lemma}
	\begin{proof}
		Let $(u,v) \in H^1 \times H^1$ be such that $\|u\|^2_{L^2} = a$ and $\|v\|^2_{L^2}=b$. Thanks to the Gagliardo-Nirenberg inequality \eqref{GN} and estimating as in \eqref{GN-app}, we have
		\[
		|P(u,v)| \leq  C(\kappa, a,b) [K(u,v)]^{\frac{d}{4}}. 
		\]
		where $C(\kappa, a,b)$ depends only on $\kappa, a,b$ and not on $u$ and $v$. Thus
		\[
		E(u,v) = \frac{1}{2} K(u,v) - P(u,v) \geq \frac{1}{2} K(u,v) - C(\kappa, a,b) [K(u,v)]^{\frac{d}{4}}.
		\]
		Since $d\leq 3$, it follows that $E(u,v) >-\infty$. To see $I(a,b)<0$, let us choose $(u,v) \in H^1 \times H^1$ such that $\|u\|^2_{L^2}=a$, $\|v\|^2_{L^2}=b$ and $u(x), v(x)>0$ for all $x \in \mathbb{R}^d$. For $\gamma>0$, set $u_\gamma(x) = \gamma^{\frac{d}{2}} u(\gamma x)$ and $v_\gamma(x)= \gamma^{\frac{d}{2}} v(\gamma x)$. It is easy to see that for all $\gamma>0$, $\|u_\gamma\|^2_{L^2} = \|u\|^2_{L^2}=a$ and $\|v_\gamma\|^2_{L^2} = \|v\|^2_{L^2}=b$. Moreover,
		\[
		E(u_\gamma, v_\gamma) = \frac{1}{2} K(u_\gamma, v_\gamma) - P(u_\gamma,v_\gamma) = \frac{\gamma^2}{2} K(u,v) - \gamma^{\frac{d}{2}}  P(u,v).
		\]
		Since $d\leq 3$ and $P(u,v)>0$, by taking $\gamma$ sufficiently small, we obtain $E(u_\gamma, v_\gamma)<0$. The proof is complete.
	\end{proof}
	
	\begin{lemma} \label{lem-pro-min}
		Let $d\leq 3$ and $a,b>0$. Let $(u_n,v_n)_{n\geq 1}$ be a minimizing sequence for $I(a,b)$. Then there exist $C>0$, $\rho>0$ and $\eta>0$ such that for sufficiently large $n$:
		\begin{itemize}
			\item[(1)] $K(u_n,v_n) \leq C$;
			\item[(2)] $\|\nabla u_n\|_{L^2} \geq \rho$ and $\|\nabla v_n\|_{L^2} \geq \eta$.
		\end{itemize}
	\end{lemma}
	\begin{proof}
	By \eqref{GN-app}, we have
	\begin{align*}
	\frac{1}{2} K(u_n,v_n) &= E(u_n,v_n) + P(u_n,v_n) \\
	& \leq E(u_n,v_n) + C(\kappa) [M(u_n,v_n)]^{\frac{3}{2}-\frac{d}{4}} [K(u_n,v_n)]^{\frac{d}{4}}.
	\end{align*}
	Since $\|u_n\|^2_{L^2} \rightarrow a, \|v_n\|^2_{L^2} \rightarrow b$ and $E(u_n, v_n) \rightarrow I(a,b)$ as $n\rightarrow \infty$, it follows that for sufficiently large $n$, $M(u_n,v_n) \leq C$ and $E(u_n, v_n) \leq C$ for some constant $C>0$. Note that the constant $C>0$ may change from line to line. Thus
	\[
	\frac{1}{2} K(u_n,v_n) \leq C + C(\kappa) [K(u_n,v_n)]^{\frac{d}{4}}.
	\]
	Since $d\leq 3$, we get $K(u_n,v_n) \leq C$ for some $C>0$ and sufficiently large $n$. This proves the first item. Note that this item implies that every minimizing sequence for $I(a,b)$ is bounded in $H^1 \times H^1$.
	
	Let us show $\|\nabla u_n\|_{L^2} \geq \rho$. Assume by contradiction that up to a subsequence, $\lim_{n\rightarrow \infty} \|\nabla u_n\|_{L^2} =0$. 
	By the H\"older inequality and Sobolev embedding, we have 
	\[
	|P(u_n,v_n)| \leq C \|u_n\|^2_{L^4} \|v_n\|_{L^2} \leq C(b) \|\nabla u_n\|_{L^2}^2 \rightarrow 0 \text{ as } n\rightarrow \infty.
	\]
	In particular, 
	\[
	I(a,b) = \lim_{n\rightarrow \infty} E(u_n,v_n) = \lim_{n\rightarrow \infty} \frac{\kappa}{2} \|\nabla v_n\|^2_{L^2} \geq 0,
	\]
	which is a contradiction to Lemma $\ref{lem-neg}$. It follows that there exists $\rho>0$ such that $\|\nabla u_n\|_{L^2} \geq \rho$. The proof of $\|\nabla v_n\|_{L^2} \geq \eta$ is similar, we omit the details.
	\end{proof}
	
	Now to each minimizing sequence ${(u_n,v_n)}_{n\geq 1}$ for $I(a,b)$, we associate with the following sequence of nondecreasing functions (L\'evy concentration functions) $M_n: [0,\infty) \rightarrow [0, a+b]$ defined by
	\[
	M_n(R) := \sup_{y \in \mathbb{R}^d} \int_{B(y,R)} |u_n(x)|^2 + |v_n(x)|^2 dx.
	\]
	Since $\|u_n\|^2_{L^2} \rightarrow a$ and $\|v_n\|^2_{L^2} \rightarrow b$ as $n\rightarrow \infty$, then $(M_n)_{n\geq 1}$ is a uniformly bounded sequence of nondecreasing functions on $[0,\infty)$. By Helly's selection theorem, we see that $(M_n)_{n\geq 1}$ must have a subsequence, still denoted by $(M_n)_{n\geq 1}$, that converges pointwise and uniformly on compact sets to a nonnegative nondecreasing function $M: [0,\infty) \rightarrow [0, a+b]$. Let
	\begin{align}\label{def-L}
	L:= \lim_{R\rightarrow \infty} M(R) = \lim_{R\rightarrow \infty} \lim_{n\rightarrow \infty} \sup_{y \in \mathbb{R}^d} \int_{B(y,R)} |u_n(x)|^2 + |v_n(x)|^2 dx. 
	\end{align}
	We have $0 \leq L \leq a+b$. By the Lions' concentration compactness lemma (\cite[Lemma I.1]{Lions1}), there are three (mutually exclusive) possibilities for the value of $L$:
	\begin{itemize}
		\item[(1)] {\bf (Vanishing)} $L=0$. Since $M(R)$ is non-negative and nondecreasing, it follows that
		\[
		M(R) = \lim_{n\rightarrow \infty} \sup_{y \in \mathbb{R}^d} \int_{B(y,R)} |u_n(x)|^2 + |v_n(x)|^2 dx =0,
		\]
		for every $R \in [0,\infty)$;
		\item[(2)] {\bf (Dichotomy)} $L \in (0, a+b)$;
		\item[(3)] {\bf (Compactness)} $L= a+b$. In this case, there exists a sequence $(y_n)_{n\geq 1} \subset \mathbb{R}^d$ such that $|u_n(\cdot+y_n)|^2+ |v_n(\cdot+y_n)|^2$ is tight, that is, for all $\epsilon >0$, there exists $R(\epsilon)>0$ such that 
		\[
		\int_{B(y_n,R(\epsilon))} |u_n(x)|^2 + |v_n(x)|^2 dx \geq (a+b) -\epsilon,
		\]
		for sufficiently large $n$.
	\end{itemize}
	
	Let us start by rulling out the ``vanishing" possibility. To do so, we recall the following well-known result (see \cite[Lemma I.1]{Lions2}).
	\begin{lemma} \label{lem-lions}
		Let $2<q<2^*$, where $2^*=\infty$ for $d=1,2$ and $2^*=\frac{2d}{d-2}$ for $d\geq 3$. Assume that $(f_n)_{n\geq 1}$ is a bounded sequence in $H^1(\mathbb{R}^d)$ and satisfies
		\[
		\sup_{y\in \mathbb{R}^d} \int_{B(y,R)} |f_n(x)|^2 dx \rightarrow 0 \text{ as } n \rightarrow \infty,
		\]
		for some $R>0$. Then $f_n \rightarrow 0$ in $L^q$.
	\end{lemma} 
	\begin{lemma} \label{lem-rul-van}
		Let $d\leq 3$ and $a, b>0$. Then for every minimizing sequence $(u_n,v_n)_{n\geq 1}$ for $I(a,b)$, it holds that $L>0$. 
	\end{lemma}
	\begin{proof}
		By contradiction, we suppose that $L=0$. Then there exist $R_0>0$ and a subsequence of $(u_n,v_n)_{n\geq 1}$, still denoted by $(u_n,v_n)_{n\geq 1}$, such that
		\[
		\sup_{y\in \mathbb{R}^d} \int_{B(y,R_0)} |u_n(x)|^2 + |v_n(x)|^2 dx \rightarrow 0 \text{ as } n\rightarrow \infty.
		\]
		Using the fact that $(u_n,v_n)_{n\geq 1}$ is bounded in $H^1 \times H^1$ (by Item (1) of Lemma $\ref{lem-pro-min})$ and Lemma $\ref{lem-lions}$, we see that $\|u_n\|_{L^q}, \|v_n\|_{L^q} \rightarrow 0$ as $n\rightarrow \infty$ with $2<q<2^*$. In particular, we have
		\[
		|P(u_n,v_n)| \leq \|u_n\|^2_{L^3} \|v_n\|_{L^3} \rightarrow 0 \text{ as } n\rightarrow \infty.
		\]
		Therefore,
		\[
		I(a,b) = \lim_{n\rightarrow \infty} E(u_n,v_n) = \lim_{n\rightarrow \infty} \frac{1}{2} K(u_n,v_n) \geq 0,
		\]
		which contradicts to Lemma $\ref{lem-neg}$. The proof is complete.
	\end{proof}
	
	We next rule out the ``dichotomy" possibility. To do so, we need the following lemmas. 
	\begin{lemma} \label{lem-ine-1}
		For $a,b>0$, it holds that
		\[
		I(a,b) < I(a,0) + I(0,b).
		\]
	\end{lemma} 
	\begin{proof}
		We first observe that $I(a,0), I(0,b) \geq 0$. Indeed, let $(u^1_n,v_n^1)_{n\geq 1}$ be a minimizing sequence for $I(a,0)$, i.e 
		\[
		\lim_{n\rightarrow \infty} \|u^1_n\|^2_{L^2}=a, \quad \lim_{n\rightarrow \infty} \|v^1_n\|^2_{L^2}=0 \text{ and } \lim_{n\rightarrow \infty} E(u^1_n,v^1_n)= I(a,0).
		\]
		Since $(u^1_n,v^1_n)_{n\geq 1}$ is a bounded sequence in $H^1 \times H^1$, the H\"older inequality and Sobolev embedding imply that
		\[
		|P(u^1_n,v^1_n)| \leq C \|u^1_n\|^2_{L^4} \|v^1_n\|_{L^2} \leq C \|\nabla u^1_n\|^2_{L^2} \|v^1_n\|_{L^2} \rightarrow 0 \text{ as } n\rightarrow \infty.
		\]
		Thus
		\[
		I(a,0) = \lim_{n\rightarrow \infty} E(u^1_n, v^1_n) = \lim_{n\rightarrow \infty} \frac{1}{2} K(u^1_n, v^1_n) \geq 0.
		\]
		The lemma follows immediately since $I(a,b)<0$ by Lemma $\ref{lem-neg}$. 
	\end{proof}
	\begin{lemma} \label{lem-ine-2}
		Let $a_1, b_1, a_2, b_2>0$. Then
		\begin{align*}
		I(a_1+a_2, b_1) &< I(a_1,b_1) + I(a_2, 0), \\
		I(a_2, b_1 +b_2) &< I(0, b_1) + I(a_2, b_2).
		\end{align*}
	\end{lemma}
	\begin{proof}
		We only give the proof for the first item, the second one is similar. Let $(u^1_n, v^1_n)_{n\geq 1}$ and $(u^2_n, v^2_n)_{n\geq 1}$ be minimizing sequences for $I(a_1,b_1)$ and $I(a_2,0)$, i.e. 
		\begin{align*}
		\lim_{n\rightarrow \infty} \|u^1_n\|^2_{L^2} &=a_1, &\quad \lim_{n\rightarrow \infty} \|v^1_n\|^2_{L^2} &= b_1, &\quad \lim_{n\rightarrow \infty} E(u^1_n,v^1_n) &= I(a_1,b_1), \\
		\lim_{n\rightarrow \infty} \|u^2_n\|^2_{L^2} &=a_2, &\quad \lim_{n\rightarrow \infty} \|v^2_n\|^2_{L^2} &= 0, &\quad \lim_{n\rightarrow \infty} E(u^2_n,v^2_n) &= I(a_2,0).
		\end{align*}
		We look for a sequence of functions $(u_n,v_n)_{n\geq 1}$ in $H^1 \times H^1$ such that $\lim_{n\rightarrow \infty} \|u_n\|^2_{L^2}=a_1+a_2, \lim_{n\rightarrow \infty} \|v_n\|^2_{L^2} = b_1$ and $\lim_{n\rightarrow \infty} E(u_n,v_n) = I(a_1+a_2, b_1)$ such that $I(a_1+a_2,b_1) <I(a_1,b_1) + I(a_2,0)$. 
		
		To do so, without loss of generality, we may assume that $u^1_n, u^2_n$ and $v^1_n, v^2_n$ are nonnegative. Since $(u^1_n, v^1_n)_{n\geq 1}$ and $(u^2_n, v^2_n)_{n\geq 1}$ are bounded in $H^1 \times H^1$, up to subsequence, we consider the values 
		\begin{align*}
		A_1 &= \frac{1}{a_2} \lim_{n\rightarrow \infty} \left( \frac{1}{2} \|\nabla u^2_n\|^2_{L^2} - \int v^1_n (u^2_n)^2 dx \right), \\
		A_2 &= \frac{1}{a_1} \lim_{n\rightarrow \infty} \left( \frac{1}{2} \|\nabla u^1_n\|^2_{L^2} - \int v^1_n (u^1_n)^2dx\right).
		\end{align*}
		Assume first $A_2>A_1$. Set $\alpha_{11}=1 +\frac{a_1}{a_2}$. It follows that $\|\sqrt{\alpha_{11}} u^2_n\|^2_{L^2}=a_1+a_2, \|v^1\|^2_{L^2} = b_1$. By the non-negativity of $u^1_n, u^2_n$ and $v^1_n, v^2_n$, we have
		\begin{align} \label{cas-1}
		I(a_1+a_2,b_1) &\leq E(\sqrt{\alpha_{11}} u^2_n, v^1_n) \nonumber\\
		&= \frac{1}{2} K(\sqrt{\alpha_{11}} u^2_n, v^1_n) - \int v^1_n (\sqrt{\alpha_{11}} u^2_n)^2 dx \nonumber\\
		&= \frac{\alpha_{11}}{2} \|\nabla u^2_n\|^2_{L^2} + \frac{\kappa}{2} \|\nabla v^1_n\|^2_{L^2} - \alpha_{11} \int v_n^1 (u^2_n)^2 dx \nonumber \\
		&= \frac{1}{2} \|\nabla u^2_n\|^2_{L^2}  + \frac{\kappa}{2} \|\nabla v^1_n\|^2_{L^2} - \int v_n^1 (u^2_n)^2 dx + \frac{a_1}{a_2} \left( \frac{1}{2} \|\nabla u^2_n\|^2_{L^2} - \int v^1_n (u^2_n)^2 dx \right) \\
		&= \frac{1}{2} K(u^2_n, v^2_n) + \frac{\kappa}{2} \|\nabla v^1_n\|^2_{L^2} + \frac{a_1}{a_2} \left( \frac{1}{2} \|\nabla u^2_n\|^2_{L^2} - \int v^1_n (u^2_n)^2 dx \right) \nonumber \\
		& \mathrel{\phantom{= \frac{1}{2} K(u^2_n, v^2_n) + \frac{\kappa}{2} \|\nabla v^1_n\|^2_{L^2}}} - \int v_n^1 (u^2_n)^2 dx -\frac{\kappa}{2} \|\nabla v^2_n\|^2_{L^2} \nonumber \\
		&\leq \frac{1}{2} K(u^2_n, v^2_n) + \frac{\kappa}{2} \|\nabla v^1_n\|^2_{L^2} + \frac{a_1}{a_2} \left( \frac{1}{2} \|\nabla u^2_n\|^2_{L^2} - \int v^1_n (u^2_n)^2 dx \right). \nonumber
		\end{align}
		Set $\delta = a_1(A_2-A_1)>0$. Passing the limit as $n\rightarrow \infty$ in the above inequality and note that $I(a_2,0) =\lim_{n\rightarrow \infty} \frac{1}{2} K(u^2_n, v^2_n)$, we get
		\begin{align*}
		I(a_1+a_2,b_1) &\leq I(a_2,0) + \lim_{n\rightarrow \infty} \frac{\kappa}{2} \|\nabla v^1_n\|^2_{L^2} + \frac{a_1}{a_2}(a_2 A_1) \\
		& = I(a_2,0) + \lim_{n\rightarrow \infty} \frac{\kappa}{2} \|\nabla v^1_n\|^2_{L^2} + a_1 A_2 -\delta \\
		&= I(a_2,0) + \lim_{n\rightarrow \infty} E(u^1_n, v^1_n) -\delta \\
		&= I(a_2,0) + I(a_1, b_1) -\delta < I(a_1,b_1) + I(a_2,0).
		\end{align*}
		In the case $A_1 >A_2$. Set $\alpha_{11} = 1+ \frac{a_2}{a_1}$. It follows that
		\begin{align*}
		I(a_1+a_2,b_1) & \leq E(\sqrt{\alpha_{11}} u^1_n, v^1_n) \\
		&= \frac{\alpha_{11}}{2} \|\nabla u^1_n\|^2_{L^2} +\frac{\kappa}{2} \|\nabla v^1_n\|^2_{L^2} - \alpha_{11} \int v^1_n (u^2_n)^2 dx \\
		&= E(u^1_n, v^1_n) +\frac{a_2}{a_1} \left( \frac{1}{2} \|\nabla u^1_n\|^2_{L^2} - \int v^1_n (u^1_n)^2 dx \right) \\
		& \leq E(u^1_n, v^1_n) +\frac{a_2}{a_1} \left( \frac{1}{2} \|\nabla u^1_n\|^2_{L^2} - \int v^1_n (u^1_n)^2 dx \right)  + \frac{\kappa}{2} \|\nabla v^2_n\|^2_{L^2} + \int v^1_n (u^2_n)^2 dx. 
		\end{align*} 
		Set $\delta = a_2 (A_1 -A_2) >0$. Passing the limit as $n\rightarrow \infty$, we obtain
		\begin{align*}
		I(a_1+a_2, b_1) &\leq I(a_1,b_1)  + \frac{a_2}{a_1}(a_1A_2) + \lim_{n\rightarrow \infty} \frac{\kappa}{2} \|\nabla v^2_n\|^2_{L^2} + \lim_{n\rightarrow \infty} \int v^1_n (u^2_n)^2 dx  \\
		&= I(a_1,b_1) + \lim_{n\rightarrow \infty} \frac{1}{2} \|\nabla u^2_n\|^2_{L^2} - \lim_{n\rightarrow \infty} \int v^1_n (u^2_n)^2 dx - \delta \\
		&\mathrel{\phantom{= I(a_1,b_1) + \lim_{n\rightarrow \infty} \frac{1}{2} \|\nabla u^2_n\|^2_{L^2}}} + \lim_{n\rightarrow \infty} \frac{\kappa}{2} \|\nabla v^2_n\|^2_{L^2} + \lim_{n\rightarrow \infty} \int v^1_n (u^2_n)^2 dx\\
		&= I(a_1,b_1) + \lim_{n\rightarrow \infty} \frac{1}{2} K(u^2_n, v^2_n) -\delta \\
		&= I(a_1, b_1) + I(a_2,0)-\delta< I(a_1,b_1) + I(a_2,0).
		\end{align*}
		Finally, we consider the case $A_1=A_2$. Set $\alpha_{11} = 1+ \frac{a_1}{a_2}$. By \eqref{cas-1}, we have
		\begin{align*}
		I(a_1+a_2,b_1) &\leq E(\sqrt{\alpha_{11}} u^2_n, v^1_n) \\
		&= \frac{\alpha_{11}}{2} \|\nabla u^2_n\|^2_{L^2} +\frac{\kappa}{2} \|\nabla v^1_n\|^2_{L^2} - \alpha_{11} \int v^1_n (u^2_n)^2 dx \\
		&= \frac{1}{2} K(u^2_n, v^2_n) + \frac{\kappa}{2} \|\nabla v^1_n\|^2_{L^2} + \frac{a_1}{a_2} \left(\frac{1}{2} \|\nabla u^2_n\|^2_{L^2} - \int v^1_n (u^2_n)^2 dx \right)  \\
		& \mathrel{\phantom{= \frac{1}{2} K(u^2_n, v^2_n) + \frac{\kappa}{2} \|\nabla v^1_n\|^2_{L^2}}} - \int v_n^1 (u^2_n)^2 dx -\frac{\kappa}{2} \|\nabla v^2_n\|^2_{L^2}.
		\end{align*}
		We claim that there exists $\delta>0$ such that for sufficiently large $n$, 
		\[
		\int v_n^1 (u^2_n)^2 dx > \delta.
		\]
		Indeed, suppose that $\lim_{n\rightarrow \infty} \int v_n^1 (u^2_n)^2 dx =0$. Since $v^1_n, u^2_n$ are nonnegative, we have $v^1_n (u^2_n)^2 \rightarrow 0$ amost everywhere in $\mathbb{R}^d$. It contradicts to the fact $\lim_{n\rightarrow \infty} \|v^1_n\|^2_{L^2}=b_1>0$ and $\lim_{n\rightarrow \infty} \|u^2_n\|^2_{L^2}=a_2>0$.
		
		Passing the limit as $n\rightarrow \infty$, we obtain
		\begin{align*}
		I(a_1+a_2,b_1) &\leq I(a_2,0) + \lim_{n\rightarrow \infty} \frac{\kappa}{2} \|\nabla v^1_n\|^2_{L^2} + \frac{a_1}{a_2} (a_2A_1) - \delta \\
		& = I(a_2,0) + \lim_{n\rightarrow \infty} \frac{\kappa}{2} \|\nabla v^1_n\|^2_{L^2} + a_1A_2 -\delta \\
		& = I(a_2,0) + \lim_{n\rightarrow \infty} E(u^1_n, v^1_n) -\delta \\
		&= I(a_2,0) + I(a_1,b_1) -\delta < I(a_1,b_1) + I(a_2,0).
		\end{align*}
		The proof is complete.
	\end{proof}
	
	\begin{lemma} \label{lem-ine-3}
	Let $a_1, b_1, a_2, b_2>0$. Then
	\[
	I(a_1+a_2, b_1 + b_2) < I(a_1, b_1) + I(a_2, b_2).
	\]
	\end{lemma}
	\begin{proof}
	Let $(u^1_n,v^1_n)_{n\geq 1}$ and $(u^2_n, v^2_n)_{n\geq 1}$ be minimizing sequences for $I(a_1,b_1)$ and $I(a_2,b_2)$, i.e.
	\begin{align*}
	\lim_{n\rightarrow \infty} \|u^1_n\|^2_{L^2} &=a_1, &\quad \lim_{n\rightarrow \infty} \|v^1_n\|^2_{L^2} &= b_1, &\quad \lim_{n\rightarrow \infty} E(u^1_n,v^1_n) &= I(a_1,b_1), \\
	\lim_{n\rightarrow \infty} \|u^2_n\|^2_{L^2} &=a_2, &\quad \lim_{n\rightarrow \infty} \|v^2_n\|^2_{L^2} &= b_2, &\quad \lim_{n\rightarrow \infty} E(u^2_n,v^2_n) &= I(a_2,b_2).
	\end{align*}
	We look for a sequence of functions $(u_n,v_n)_{n\geq 1}$ in $H^1 \times H^1$ such that $\lim_{n\rightarrow \infty} \|u_n\|^2_{L^2} = a_1 + a_2$, $\lim_{n\rightarrow \infty} \|v_n\|^2_{L^2} = b_1+b_2$ and $\lim_{n\rightarrow \infty} E(u_n,v_n) = I(a_1+a_2, b_1+b_2)$ such that $I(a_1+a_2, b_1+b_2) < I(a_1,b_1) + I(a_2,b_2)$. 
	
	Without loss of generality, we may assume that $u^1_n, u^2_n$ and $v^1_n, v^2_n$ are nonnegative. By a density argument, we may also suppose that $u^1_n, u^2_n$ and $v^1_n, v^2_n$ have compact support. For each $n$, we choose $x_n \in \mathbb{R}^d$ so that $\tilde{v}^1_n(\cdot) = v^1_n(\cdot-x_n)$  and $v^2_n$ have disjoint support. Since $(u^1_n, v^1_n)_{n\geq 1}$ and $(u^2_n, v^2_n)_{n\geq 1}$ are bounded in $H^1 \times H^1$, passing to subsequences, we can consider the following values
	\begin{align*}
	A_1 &=  \frac{1}{a_1} \lim_{n\rightarrow \infty} \left( \frac{\|\nabla u^1_n\|^2_{L^2}}{2} - \int v^1_n(u^1_n)^2 dx \right), \\
	A_2 &= \frac{1}{a_2} \lim_{n\rightarrow \infty} \left( \frac{\|\nabla u^2_n\|^2_{L^2}}{2}-\int v^2_n (u^2_n)^2 dx \right).
	\end{align*}
	Let $v_n= \tilde{v}^1_n +v^2_n$. Since $\tilde{v}^1_n$, $v^2_n$ have disjoint support, we have that $\lim_{n\rightarrow \infty} \|v_n\|^2_{L^2} = b_1+b_2$. 
	
	Assume first $A_2>A_1$. Set $\alpha_{11}= 1+
	\frac{a_2}{a_1}$ and $\tilde{u}^1_n(\cdot)= u^1_n(\cdot-x_n)$. It follows that
	\begin{align} \label{ine-3-prof}
	I(a_1+a_2, b_1+b_2) &\leq E(\sqrt{\alpha_{11}} \tilde{u}^1_n, v_n) \nonumber \\
	&= \frac{\alpha_{11}}{2} \|\nabla \tilde{u}^1_n\|^2_{L^2} + \frac{\kappa}{2} \|\nabla v_n\|^2_{L^2} - \alpha_{11} \int v_n (\tilde{u}^1_n)^2 dx \nonumber \\
	&= \frac{\alpha_{11}}{2} \|\nabla \tilde{u}^1_n\|^2_{L^2} + \frac{\kappa}{2} \|\nabla \tilde{v}^1_n\|^2_{L^2} + \frac{\kappa}{2} \|\nabla v^2_n\|^2_{L^2}  - \alpha_{11} \int \tilde{v}^1_n (\tilde{u}^1_n)^2 dx \\
	&\mathrel{\phantom{= \frac{\alpha_{11}}{2} \|\nabla u^1_n\|^2_{L^2} + \frac{\kappa}{2} \|\nabla \tilde{v}^1_n\|^2_{L^2} + \frac{\kappa}{2} \|\nabla v^2_n\|^2_{L^2}  }}- \alpha_{11} \int v^2_n (\tilde{u}^1_n)^2 dx \nonumber \\
	&\leq  E(u^1_n, v^1_n) + \frac{\kappa}{2} \|\nabla v^2_n\|^2_{L^2} + \frac{a_2}{a_1}\left( \frac{\|\nabla u^1_n\|^2_{L^2}}{2} - \int v^1_n (u^1_n)^2 dx \right). \nonumber
	\end{align}
	Set $\delta = a_2(A_2-A_1)>0$. Passing the limit as $n\rightarrow \infty$, we obtain
	\begin{align*}
	I(a_1+a_2,b_1+b_2) &\leq I(a_1,b_1) +\lim_{n\rightarrow \infty} \frac{\kappa}{2} \|\nabla v^2_n\|^2_{L^2} + \frac{a_2}{a_1} (a_1A_1) \\
	&= I(a_1,b_1) + \lim_{n\rightarrow \infty} \frac{\kappa}{2} \|\nabla v^2_n\|^2_{L^2} + a_2 A_2 -\delta \\
	&= I(a_1,b_1) + \lim_{n\rightarrow \infty} E(u^2_n, v^2_n) -\delta \\
	&= I(a_1,b_1) + I(a_2,b_2) -\delta < I(a_1,b_1) + I(a_2,b_2).
	\end{align*}
	
	Let us now consider the case $A_1>A_2$. Set $\alpha_{11} = 1+\frac{a_1}{a_2}$. We have
	\begin{align*}
	I(a_1+a_2, b_1+b_2) &\leq E(\sqrt{\alpha_{11}} u^2_n, v_n) \\
	&= \frac{\alpha_{11}}{2} \|\nabla u^2_n\|^2_{L^2} + \frac{\kappa}{2} \|\nabla v_n\|^2_{L^2} - \alpha_{11} \int v_n (u^2_n)^2 dx \\
	&= \frac{\alpha_{11}}{2} \|\nabla u^2_n\|^2_{L^2} + \frac{\kappa}{2} \|\nabla \tilde{v}^1_n\|^2_{L^2} + \frac{\kappa}{2} \|\nabla v^2_n\|^2_{L^2} - \alpha_{11} \int \tilde{v}^1_n (u^2_n)^2 dx \\
	&\mathrel{\phantom{= \frac{\alpha_{11}}{2} \|\nabla u^2_n\|^2_{L^2} + \frac{\kappa}{2} \|\nabla \tilde{v}^1_n\|^2_{L^2} + \frac{\kappa}{2} \|\nabla v^2_n\|^2_{L^2} }} - \alpha_{11} \int v^2_n (u^2_n)^2 dx \\
	& \leq E(u^2_n, v^2_n) + \frac{\kappa}{2} \|\nabla v^1_n\|^2_{L^2} + \frac{a_1}{a_2} \left( \frac{\|\nabla u^2_n\|^2_{L^2}}{2} - \int v^2_n (u^2_n)^2 dx \right).
	\end{align*}
	Set $\delta = a_1(A_1-A_2)>0$. Passing the limit as $n\rightarrow \infty$, we get
	\begin{align*}
	I(a_1+a_2,b_1+b_2) &\leq I(a_2,b_2) + \lim_{n\rightarrow \infty} \frac{\kappa}{2} \|\nabla v^1_n\|^2_{L^2} + \frac{a_1}{a_2}(a_2A_2) \\
	&= I(a_2,b_2) + \lim_{n\rightarrow \infty} \frac{\kappa}{2} \|\nabla v^1_n\|^2_{L^2} + a_1 A_1 - \delta \\
	&= I(a_2,b_2) + \lim_{n\rightarrow \infty} E(u^1_n, v^1_n) -\delta \\
	&= I(a_2,b_2) + I(a_1,b_1) -\delta < I(a_1,b_1) + I(a_2,b_2).
	\end{align*}
	
	Finally, we consider the case $A_1=A_2$. Set $\alpha_{11}=1+ \frac{a_2}{a_1}$. By \eqref{ine-3-prof}, we have
	\begin{align*}
	I(a_1+a_2, b_1+b_2) &\leq E(\sqrt{\alpha_{11}} \tilde{u}^1_n, v_n) \\
	&= \frac{\alpha_{11}}{2} \|\nabla \tilde{u}^1_n\|^2_{L^2} + \frac{\kappa}{2} \|\nabla \tilde{v}^1_n\|^2_{L^2} + \frac{\kappa}{2} \|\nabla (v^2_n)^*\|^2_{L^2}  - \alpha_{11} \int v^1_n (u^1_n)^2 dx \\
	&\mathrel{\phantom{= \frac{\alpha_{11}}{2} \|\nabla \tilde{u}^1_n\|^2_{L^2} + \frac{\kappa}{2} \|\nabla \tilde{v}^1_n\|^2_{L^2} + \frac{\kappa}{2} \|\nabla (v^2_n)^*\|^2_{L^2}  }}- \alpha_{11} \int v^2_n (\tilde{u}^1_n)^2 dx \\
	& = E(u^1_n,v^1_n)+\frac{\kappa}{2} \|\nabla v^2_n\|^2_{L^2} + \frac{a_2}{a_1} \left( \frac{\|\nabla u^1_n\|^2_{L^2}}{2}-\int v^1_n (u^1_n)^2dx \right) \\
	&\mathrel{\phantom{= E(u^1_n,v^1_n)+\frac{\kappa}{2} \|\nabla v^2_n\|^2_{L^2} }} - \alpha_{11} \int v^2_n (\tilde{u}^1_n)^2 dx. 
	\end{align*}
	As in the proof of Lemma $\ref{lem-ine-2}$, there exists $\delta>0$ such that $\lim_{n\rightarrow \infty} \int v^2_n (\tilde{u}^1_n)^2 dx >\delta$. Passing the limit as $n\rightarrow \infty$, we get
	\begin{align*}
	I(a_1+a_2, b_1+b_2) &\leq I(a_1,b_1) + \lim_{n\rightarrow \infty} \frac{\kappa}{2} \|\nabla v^2_n\|^2_{L^2} + \frac{a_2}{a_1}(a_1A_1) -\delta \\
	&= I(a_1,b_1) + \lim_{n\rightarrow \infty} \frac{\kappa}{2} \|\nabla v^2_n\|^2_{L^2} + a_2 A_2 -\delta \\
	&= I(a_1,b_1) + \lim_{n\rightarrow \infty} E(u^2_n, v^2_n) -\delta \\
	&= I(a_1,b_1) + I(a_2,b_2)-\delta < I(a_1,b_1) + I(a_2,b_2).
	\end{align*}
	The proof is now complete.
	\end{proof}
	
	Combining Lemmas $\ref{lem-ine-1}, \ref{lem-ine-2}, \ref{lem-ine-3}$, we have the subadditivity property of $I(a,b)$.
	\begin{corollary} \label{coro-sub-add}
		Let $a_1,a_2,b_1,b_2 \geq 0$ be such that $a_1+a_2>0, b_1+b_2>0, a_1+b_1>0$ and $a_2+b_2>0$. Then it holds that
		\[
		I(a_1+a_2,b_1+b_2) < I(a_1,b_1) + I(a_2,b_2).
		\]
	\end{corollary}
	
	\begin{lemma} \label{lem-dic-key}
		Let $L$ be as in \eqref{def-L}. Let $a,b>0$ and $(u_n,v_n)_{n\geq 1}$ be a minimizing sequence for $I(a,b)$. Then there exists $(a_1,b_1) \in [0,a] \times [0,b]$ such that $L= a_1+b_1$ and
		\begin{align} \label{rev-sub-add}
		I(a_1,b_1) + I(a-a_1,b-b_1) \leq I(a,b).
		\end{align}
	\end{lemma}
	\begin{proof}
		Let $\epsilon>0$ be arbitrary. It follows from the definition of $L$ that there exist $R_\epsilon>0$ and $N_\epsilon \in \mathbb{N}$ such that for $R\geq R_\epsilon$ and $n\geq N_\epsilon$, one has $L-\epsilon <M(R) \leq M(2R) \leq L$ and $L-\epsilon < M_n(R) \leq M_n(2R) \leq L+\epsilon$. Thus by the definition of $M_n$, for every $n\geq N_\epsilon$, there exists a sequence of points $(y_n)_{n\geq 1} \subset \mathbb{R}^d$  such that
		\begin{align} \label{dic-1}
		\int_{B(y_n,R)} |u_n(x)|^2 + |v_n(x)|^2 dx > L-\epsilon \text{ and } \int_{B(y_n,2R)} |u_n(x)|^2 + |v_n(x)|^2 dx <L+\epsilon.
		\end{align}
		Now let $\vartheta \in C^\infty_0(B(0,2))$ be such that $\vartheta \equiv 1$ on $B(0,1)$ and $\chi \in C^\infty(\mathbb{R}^d)$ be such that $\vartheta^2 + \chi^2 =1$ on $\mathbb{R}^d$. For any $R>0$, we define
		\[
		\vartheta_R(x) = \vartheta(x/R), \quad \chi_R(x) = \chi(x/R).
		\]
		We next define the functions
		\begin{align*}
		(u_n^1(x), v_n^1(x)) &:= \vartheta_R(x-y_n) (u_n(x),v_n(x)), \quad x \in \mathbb{R}^d, \\
		(u^2_n(x), v^2_n(x)) &:= \chi_R(x-y_n) (u_n(x), v_n(x)), \quad x \in \mathbb{R}^d.
		\end{align*}
		Since $(u_n^i, v^i_n)_{n\geq 1}, i=1,2$ are bounded in $L^2$. Up to subsequence, we see that $\|u^1_n\|^2_{L^2} \rightarrow a_1, \|v^1_n\|^2_{L^2} \rightarrow b_1$ as $n\rightarrow \infty$, where $a_1 \in [0,a]$ and $b_1 \in [0,b]$. We also have that $\|u^2_n\|^2_{L^2} \rightarrow a-a_1, \|v^2_n\|^2_{L^2} \rightarrow b-b_1$ as $n\rightarrow \infty$. Thus
		\begin{align} \label{dic-2}
		a_1+b_1 = \lim_{n\rightarrow \infty} \int |u^1_n(x)|^2 + |v^1_n(x)|^2 dx = \lim_{n\rightarrow \infty} \int \vartheta_R^2(x-y_n) (|u_n(x)|^2 + |v_n(x)|^2) dx.
		\end{align}
		By \eqref{dic-1} and \eqref{dic-2}, we have
		\begin{align} \label{dic-con-1}
		|a_1+b_1 - L| <\epsilon.
		\end{align}
		We claim now that there exists $C>0$ such that for every $n$,
		\begin{align} \label{dic-3}
		E(u^1_n, v^1_n) + E(u^2_n,v^2_n) \leq E(u_n,v_n) + C\epsilon.
		\end{align}
		Indeed, 
		\begin{align*}
		E(u^1_n,v^1_n) &= \frac{\|\nabla u^1_n\|^2}{2} +\frac{\kappa}{2} \|\nabla v^1_n\|^2_{L^2} - \text{Re}\int v^1_n (u^1_n)^2 dx \\
		&= \frac{\|\nabla(\vartheta_R u_n)\|^2_{L^2}}{2} + \frac{\kappa}{2} \|\nabla (\vartheta_R v_n)\|^2_{L^2}  - \text{Re} \int \vartheta^3_R v_n (u_n)^2 dx.
		\end{align*}
		We see that
		\begin{align*}
		\|\nabla(\vartheta_R u_n)\|^2_{L^2} &\leq (\|\vartheta_R \nabla u_n\|^2_{L^2} + \|\nabla \vartheta_R u_n\|^2_{L^2})^2 \\
		&= \|\vartheta_R \nabla u_n\|^2_{L^2} + \|\nabla \vartheta_R u_n\|^2_{L^2} + 2 \|\vartheta_R \nabla u_n\|_{L^2} \|\nabla \vartheta_R u_n\|_{L^2} \\
		&\leq \|\vartheta_R \nabla u_n\|^2_{L^2} +  \|\nabla \vartheta_R\|_{L^\infty}^2 \|u_n\|^2_{L^2} + 2 \|\vartheta_R\|_{L^\infty} \|\nabla u_n\|_{L^2} \|\nabla \vartheta_R\|_{L^\infty} \|u_n\|_{L^2} \\
		&= \|\vartheta_R \nabla u_n\|^2_{L^2} +  \frac{\|\nabla \vartheta\|_{L^\infty}^2}{R^2} \|u_n\|^2_{L^2} + 2 \|\vartheta_R\|_{L^\infty} \frac{\|\nabla \vartheta\|_{L^\infty}}{R} \|\nabla u_n\|_{L^2} \|u_n\|_{L^2} \\
		&\leq \|\vartheta_R \nabla u_n\|^2_{L^2} + C\epsilon.
		\end{align*}
		The last inequality follows by taking $R$ sufficiently large and using the fact that $u_n$ is bounded in $H^1$. Similarly, we have
		\[
		\|\nabla(\vartheta_R v_n)\|^2_{L^2} \leq \|\vartheta_R \nabla v_n\|^2_{L^2} + C\epsilon.
		\]
		Thus
		\[
		E(u^1_n,v^1_n) \leq \frac{\|\vartheta_R \nabla u_n\|^2_{L^2}}{2} + \frac{\kappa}{2} \|\vartheta_R \nabla v_n\|^2_{L^2} - \text{Re}\int \vartheta_R^2 v_n (u_n)^2 dx + \text{Re}\int (\vartheta_R^2 - \vartheta_R^3) v_n (u_n)^2 dx + C\epsilon.
		\]
		Due to the support of $\vartheta_R$ and \eqref{dic-1}, we have
		\begin{align*}
		\left| \text{Re} \int (\vartheta_R^2 - \vartheta_R^3) v_n (u_n)^2 dx \right| &\leq C \int_{R<|x-y_n|<2R} |v_n| |u_n|^2 dx \\
		&\leq C \int_{R<|x-y_n|<2R} |v_n|^3 + |u_n|^3 dx \\
		& \leq C \int_{R<|x-y_n|<2R} |v_n|^2 + |u_n|^2 dx \leq C\epsilon.
		\end{align*}
		We obtain 
		\[
		E(u^1_n, v^1_n) \leq \int \vartheta_R^2 \left(\frac{1}{2}|\nabla u_n|^2 + \frac{\kappa}{2} |\nabla v_n|^2 - \text{Re}(v_n (u_n)^2) \right) dx + C\epsilon.
		\]
		Similarly, we have
		\[
		E(u^2_n, v^2_n) \leq \int \chi_R^2 \left(\frac{1}{2}|\nabla u_n|^2 + \frac{\kappa}{2} |\nabla v_n|^2 - \text{Re}(v_n (u_n)^2) \right) dx + C\epsilon.
		\]
		Suming these two quantities and using the fact $\vartheta_R^2 + \chi_R^2 =1$, we get \eqref{dic-3}. 
		
		We now consider the case all $a_1, b_1, a-a_1, b-b_1$ are positive. We set 
		\[
		\alpha^1_n = \frac{\sqrt{a_1}}{\|u^1_n\|_{L^2}}, \quad \beta^1_n = \frac{\sqrt{b_1}}{\|v^1_n\|_{L^2}}, \quad \alpha^2_n = \frac{\sqrt{a-a_1}}{\|u^2_n\|_{L^2}}, \quad \beta^2_n = \frac{\sqrt{b-b_1}}{\|v^2_n\|_{L^2}}. 
		\]
		It follows that
		\[
		\|\alpha^1_n u^1_n\|^2_{L^2} = a_1, \quad \|\beta^1_n v^1_n\|^2_{L^2} = b_1, \quad \|\alpha^2_n u^2_n\|^2_{L^2} = a-a_1, \quad \|\beta^2_n v^2_n\|^2_{L^2} = b-b_1.
		\]
		Note that all the scaling $\alpha^1_n, \beta^1_n, \alpha^2_n$ and $\beta^2_n$ tend to $1$ as $n\rightarrow \infty$. Thus, up to subsequence,
		\begin{align} \label{dic-4}
		I(a_1,b_1) + I(a-a_1,b-b_1) \leq \lim_{n\rightarrow \infty} E(u^1_n, v^1_n) + E(u^2_n, v^2_n).
		\end{align}
		In the case $a_1 =0$, we see that $\|u^1_n\|^2_{L^2} \rightarrow 0$ as $n\rightarrow \infty$. This implies that 
		\[
		\text{Re} \int v^1_n (u^1_n)^2 dx \rightarrow 0 \text{ as } n\rightarrow \infty.
		\]
		Therefore, 
		\[
		I(0,b_1) \leq \lim_{n\rightarrow \infty} \frac{\kappa}{2} \|\nabla v^1_n\|^2_{L^2} \leq \lim_{n\rightarrow \infty} E(u^1_n, v^1_n).
		\]
		Arguing as in the first case, we get
		\[
		I(a, b-b_1) \leq \lim_{n\rightarrow \infty} E(u^2_n, v^2_n).
		\]
		Hence we still get \eqref{dic-4}. The case $b_1=0$ is similar.
		Therefore in all cases, we have \eqref{dic-4}. Combining \eqref{dic-4} with \eqref{dic-3}, we obtain
		\begin{align} \label{dic-con-2}
		I(a_1,b_1) + I(a-a_1,b-b_1) &\leq \lim_{n\rightarrow \infty} E(u^1_n,v^1_n) + E(u^2_n,v^2_n) \nonumber \\
		&\leq \lim_{n\rightarrow \infty} E(u_n,v_n) + C \epsilon \nonumber \\
		&\leq I(a,b)+ C\epsilon.
		\end{align}
		As $\epsilon>0$ is arbitrary, the result follows from \eqref{dic-con-1} and \eqref{dic-con-2}.
	\end{proof}
	
	We are now able to rule out the ``dichotomy" possibility. 
	\begin{lemma} \label{lem-rul-dic}
		Let $d\leq 3$, $a,b>0$ and $L$ be as in \eqref{def-L}. Then for any minimizing sequence $(u_n,v_n)_{n\geq 1}$ for $I(a,b)$, it holds that $L\notin (0, a+b)$, that is, the dichotomy cannot occur.
	\end{lemma}
	
	\begin{proof}
		Assume by contradiction that the dichotomy occurs, that is, $L \in (0, a+b)$. Let $a_1,b_1$ be as in Lemma $\ref{lem-dic-key}$, i.e. $a_1+b_1=L$ and 
		\begin{align} \label{rul-dic}
		I(a_1,b_1) + I(a-a_1, b-b_1) \leq I(a,b).
		\end{align}
		On the other hand, since $a,b>0$, $a_1+b_1 =L >0$ and $a-a_1 + b-b_1 = a+b-L>0$, it follows from Corollary $\ref{coro-sub-add}$ that
		\[
		I(a,b) < I(a_1,b_1) + I(a-a_1, b-b_1)
		\]
		which contradicts to \eqref{rul-dic}. The proof is complete.
	\end{proof}

	We are now able to show the existence of minimizers for $I(a,b)$. 
	\begin{lemma} \label{lem-exi-min}
		Let $d \leq 3$ and $a,b>0$. Let $(u_n,v_n)_{n\geq 1}$ be any minimizing sequence for $I(a,b)$. Then there exists a sequence $(y_n)_{n\geq 1} \subset \mathbb{R}^d$ such that the sequence $(u_n(\cdot+ y_n), v_n(\cdot+y_n))_{n\geq 1}$ has a subsequence which converges strongly in $H^1 \times H^1$ to some $(u,v)$, which is a minimizer for $I(a,b)$. That is the set $\mathcal{G}_{a,b}$ is not empty.
	\end{lemma}
	
	\begin{proof}
		By Lemma $\ref{lem-rul-van}$ and Lemma $\ref{lem-rul-dic}$, we have $L=a+b$. Thus by the Lions' concentration compactness lemma, there exists a sequence $(y_n)_{n\geq 1} \subset \mathbb{R}^d$ such that for each $k\in \mathbb{N}$, there exists $R_k>0$ such that
		\[
		\int_{B(y_n,R_k)} |u_n(x)|^2 + |v_n(x)|^2 \geq (a+b) -\frac{1}{k},
		\]
		equivalently
		\begin{align} \label{exi-min-1}
		\int_{B(0,R_k)} |\tilde{u}_n(x)|^2 + |\tilde{v}_n(x)|^2 dx > (a+b) -\frac{1}{k},
		\end{align}
		where $\tilde{u}_n(x) = u_n(x+ y_n)$ and $\tilde{v}_n(x) = v_n(x+y_n)$. Since the translated sequence $(\tilde{u}_n, \tilde{v}_n)_{n\geq 1}$ is bounded in $H^1 \times H^1$, so up to subsequence, $(\tilde{u}_n, \tilde{v}_n) \rightharpoonup (u,v)$ in $H^1 \times H^1$. By the Fatou's lemma, we see that $\|u\|^2_{L^2} + \|v\|^2_{L^2} \leq a+b$. For each $k \in \mathbb{N}$, the embedding $H^1(B(0,R_k)) \hookrightarrow L^2(B(0,R_k))$ is compact, so up to a subsequence, we have $(\tilde{u}_n, \tilde{v}_n) \rightarrow (u,v)$ strongly in $L^2(B(0,R_k)) \times L^2(B(0,R_k))$. By a standard diagonalization argument, one may assume that there exists a subsequence of $(\tilde{u}_n, \tilde{v}_n)_{n\geq 1}$, still denoted by $(\tilde{u}_n, \tilde{v}_n)_{n\geq 1}$, satisfies
		\[
		(\tilde{u}_n, \tilde{v}_n) \rightarrow (u,v) \text{ strongly in } L^2(B(0,R_k)) \times L^2(B(0,R_k)) \text{ for every } k \in \mathbb{N}.
		\]
		Passing the limit as $n\rightarrow \infty$ in \eqref{exi-min-1}, we obtain
		\begin{align*}
		(a+b) -\frac{1}{k} &\leq \lim_{n\rightarrow \infty} \int_{B(0,R_k)} |\tilde{u}_n|^2 + |\tilde{v}_n|^2 dx \\
		&= \|u\|^2_{L^2(B(0,R_k))} + \|v\|^2_{L^2(B(0,R_k))} \leq \|u\|^2_{L^2} +\|v\|^2_{L^2}. 
		\end{align*}
		Since $\|u\|^2_{L^2}+\|v\|^2_{L^2} \leq a+b$ and $k\in \mathbb{N}$ is arbitrary, it follows that $\|u\|^2_{L^2} + \|v\|^2_{L^2} = a+b$. Thus $(\tilde{u}_n, \tilde{v}_n)_{n\geq 1} \rightarrow (u,v)$ strongly in $L^2 \times L^2$. 
		
		Since $(\tilde{u}_n, \tilde{v}_n) \rightharpoonup (u,v)$ weakly in $H^1\times H^1$, we have
		\[
		\|\nabla u\|^2_{L^2} \leq \liminf_{n\rightarrow \infty} \|\nabla \tilde{u}_n\|^2_{L^2}, \quad \|\nabla v\|^2_{L^2} \leq \liminf_{n\rightarrow \infty} \|\nabla \tilde{v}_n\|^2_{L^2}.
		\]
		On the other hand,
		\begin{align} \label{exi-min-2}
		\text{Re} \int \tilde{v}_n (\tilde{u}_n)^2 dx \rightarrow \text{Re} \int v u^2 dx \text{ as } n\rightarrow \infty.
		\end{align}
		Indeed, by H\"older's inequality and Sobolev embedding,
		\begin{align*}
		\left| \int (\tilde{v}_n) (\tilde{u}_n)^2 dx -  \int v u^2 dx \right| &\leq \left|\int (\tilde{v}_n-v) (\tilde{u}_n)^2 dx \right| + \left|\int v((\tilde{u}_n)^2 - u^2) dx \right| \\
		&\lesssim \|\tilde{v}_n-v\|_{L^2} \|\tilde{u}_n\|^2_{L^4} + \|v\|_{L^4} \|\tilde{u}_n-u\|_{L^2} \|\tilde{u}_n + u\|_{L^4} \\
		& \lesssim \|\tilde{v}_n-v\|_{L^2} \|\nabla \tilde{u}_n\|^2_{L^2} + \|\nabla v\|_{L^2} \|\tilde{u}_n-u\|_{L^2} \|\nabla(\tilde{u}_n + u)\|_{L^2} \rightarrow 0,
		\end{align*}
		as $n\rightarrow \infty$. This implies that
		\[
		E(u,v) \leq \lim_{n\rightarrow \infty} E(\tilde{u}_n, \tilde{v}_n) = I(a,b).
		\]
		On the other hand, $\|u\|^2_{L^2} = \lim_{n\rightarrow \infty} \|\tilde{u}_n\|^2_{L^2} = a$ and $\|v\|^2_{L^2}= \lim_{n\rightarrow \infty} \|\tilde{v}_n\|^2_{L^2}=b$, we have $I(a,b)\leq E(u,v)$. Therefore, $I(a,b)= E(u,v)$ or $(u,v)$ is a minimizer for $I(a,b)$ or $(u,v) \in \mathcal{G}_{a,b}$. 
		
		Finally, since $E(u,v)= \lim_{n\rightarrow \infty} E(\tilde{u}_n, \tilde{v}_n)$ and \eqref{exi-min-2}, we have
		\[
		\frac{1}{2}\|\nabla u\|^2_{L^2} + \frac{\kappa}{2} \|\nabla v\|^2_{L^2} = \lim_{n\rightarrow \infty} \frac{1}{2}\|\nabla \tilde{u}_n\|^2_{L^2} + \frac{\kappa}{2} \|\nabla \tilde{v}_n\|^2_{L^2}.
		\]
		This combined with $\|u\|^2_{L^2} + \|v\|^2_{L^2} = \lim_{n\rightarrow \infty} \|\tilde{u}_n\|^2_{L^2} + \|\tilde{v}_n\|^2_{L^2}$ implies that
		\[
		\|(u,v)\|_{H^1\times H^1} = \lim_{n\rightarrow \infty} \|(\tilde{u}_n,\tilde{v}_n)\|_{H^1\times H^1}.
		\]
		Since $(\tilde{u}_n, \tilde{v}_n) \rightarrow (u,v)$ weakly in $H^1 \times H^1$, we see that $(\tilde{u}_n, \tilde{v}_n) \rightarrow (u,v)$ strongly in $H^1 \times H^1$. The proof is complete.
	\end{proof}

	\noindent \textit{Proof of Theorem $\ref{theo-exi-min}$.}
		Item (1) follows from Lemma $\ref{lem-exi-min}$. We prove Item (2) by contradiction. Suppose that \eqref{lit-gro-1} is not true. Then there exist a subsequence $(u_{n_k}, v_{n_k})_{k\geq 1}$ of $(u_n,v_n)_{n\geq 1}$ and a constant $\delta>0$ such that
		\begin{align} \label{lit-gro-pro}
		\lim_{n\rightarrow \infty} \inf_{(w,z) \in G_{a,b}, y \in \mathbb{R}^d} \|(u_{n_k}(\cdot+y), v_{n_k}(\cdot+y))- (w,z)\|_{H^1 \times H^1} \geq \delta.
		\end{align}
		Since $(u_{n_k},v_{n_k})_{k\geq 1}$ is still a minimizing sequence for $I(a,b)$, by Item (1), there exist a sequence $(y_k)_{k\geq 1} \subset \mathbb{R}^d$ and $(g,h) \in \mathcal{G}_{a,b}$ such that up to subsequence,
		\[
		\lim_{n\rightarrow \infty} \|(u_{n_k}(\cdot+y_k), v_{n_k}(\cdot+y_k)) - (g,h)\|_{H^1\times H^1} =0,
		\]
		which contradicts to \eqref{lit-gro-pro}. Item (2) is thus proved. Item (3) follows directly from Item (2) using the fact that if $(w,z) \in \mathcal{G}_{a,b}$, then $(w(\cdot-y), z(\cdot-y)) \in \mathcal{G}_{a,b}$ for any $y \in \mathbb{R}^d$. Finally, we show Item (4). Since $(u,v)$ is a minimizer of $I(a,b)$, then there exist Lagrange multipliers $\omega_1, \omega_2 \in \mathbb{R}$ such that
		\[
		K'(u,v)[\chi,\vartheta]=0, \quad \forall (\chi, \vartheta) \in C^\infty_0 \times C^\infty_0,
		\]
		where $K(u,v)= E(u,v) + \frac{\omega_1}{2}\|u\|^2_{L^2} + \frac{\omega_2}{2}\|v\|^2_{L^2}$ and the prime denotes the Fr\'echet derivative. A direct computation shows
		\[
		\text{Re} \int \left( \nabla u \cdot \nabla \overline{\chi} + \kappa \nabla v \cdot \nabla \overline{\vartheta} + \omega_1 u \overline{\chi} + \omega_2 v \overline{\vartheta} - 2 v \overline{u} \overline{\chi} - u^2 \overline{\vartheta} \right) dx =0.
		\]
		Testing $(i\chi, i \vartheta)$ instead of $(\chi,\vartheta)$ and using the fact $\text{Re}(iz) = - \text{Im} (z)$, we obtain
		\[
		\text{Im } \int \left( \nabla u \cdot \nabla \overline{\chi} + \kappa \nabla v \cdot \nabla \overline{\vartheta} + \omega_1 u \overline{\chi} + \omega_2 v \overline{\vartheta} - 2 v \overline{u} \overline{\chi} - u^2 \overline{\vartheta} \right) dx =0.
		\]
		Thus
		\[
		\int \left( \nabla u \cdot \nabla \overline{\chi} + \kappa \nabla v \cdot \nabla \overline{\vartheta} + \omega_1 u \overline{\chi} + \omega_2 v \overline{\vartheta} - 2 v \overline{u} \overline{\chi} - u^2 \overline{\vartheta} \right) dx =0.
		\]
		Hence
		\[
		\left\{ 
		\renewcommand*{\arraystretch}{2}
		\begin{array}{rcl} 
		\mathlarger{\int} \nabla u \cdot \nabla \overline{\chi} dx + \omega_1 \mathlarger{\int} u \overline{\chi} dx & = & 2 \mathlarger{\int} v \overline{u}  \overline{\chi} dx, \\  
		\kappa \mathlarger{\int} \nabla v \cdot \nabla \overline{\vartheta} dx + \omega_2 \mathlarger{\int} v \overline{\vartheta} dx & = & \mathlarger{\int} u^2 \overline{\vartheta} dx,
		\end{array}
		\right.
		\]
		for any $(\chi, \vartheta) \in C^\infty_0 \times C^\infty_0$. This implies that $(u,v)$ solves \eqref{sys-ell-equ} in the sense of distribution. Moreover, by \cite[Lemma 1.3]{Tao}, $(u,v)$ is also a classical solution of \eqref{wor-Syst}. Since $\|\nabla (|u|)\|^2_{L^2} \leq \|\nabla u\|^2_{L^2}$ and $\text{Re} (\langle v, u^2 \rangle) \leq \langle |v|, |u|^2 \rangle$, we see that
		\[
		E(|u|,|v|) \leq E(u,v).
		\]
		This shows that $(|u|,|v|)$ is also a minimizer of $I(a,b)$ and $E(|u|,|v|)=E(u,v)$. It follows that $\|\nabla (|u|)\|^2_{L^2} = \|\nabla u\|^2_{L^2}$ and $\|\nabla(|v|)\|^2_{L^2} = \|\nabla v\|^2_{L^2}$. We now set $\tilde{u}(x) := \frac{u(x)}{|u(x)|}$. Since $|\tilde{u}|^2=1$, we have $\text{Re} (\overline{\tilde{u}} \nabla \tilde{u}) =0$ and 
		\[
		\nabla u = (\nabla(|u|)) \tilde{u} + |u| \nabla \tilde{u} = \tilde{u} (\nabla(|u|) + |u|\overline{\tilde{u}} \nabla \tilde{u}).
		\]
		Thus, we get $|\nabla u|^2 = |\nabla(|u|)|^2 + |u|^2 |\nabla \tilde{u}|^2$. Since $\|\nabla (|u|)\|^2_{L^2} = \|\nabla u\|^2_{L^2}$, it follows that
		\[
		\int |u|^2 |\nabla \tilde{u}|^2 dx =0.
		\]
		Thus $|\nabla \tilde{u}|=0$ and hence $\tilde{u}$ is a constant with $|\tilde{u}|=1$. This shows that there exists $\theta_1 \in \mathbb{R}$ such that $u(x) = e^{i\theta_1} \vartheta(x)$, where $\vartheta(x) = |u(x)|$. Similarly, there exists $\theta_2 \in \mathbb{R}$ such that $v(x)= e^{i\theta_2} \zeta(x)$, where $\zeta(x)= |v(x)|$. The proof is complete.
	\hfill $\Box$
		
	\subsection{The variational problem $J(c)$}
	\label{Sub:1.2}
	In this subsection, we study the variational problem \eqref{prob-2}.	
	\begin{lemma} \label{lem-neg-2}
		Let $d\leq 3$ and $c>0$. Then $-\infty < J(c) <0$. Moreover, any minimizing sequence of $J(c)$ is bounded in $H^1\times H^1$.
	\end{lemma}

	\begin{proof}
		Let $(u,v) \in H^1 \times H^1$ be such that $M(u,v) =c$.  By \eqref{GN-app}, we see that
		\[
		|P(u,v)| \leq C [M(u,v)]^{\frac{3}{2}-\frac{d}{4}} [K(u,v)]^{\frac{d}{4}}.
		\]
		Since $M(u,v)=c$, it follows that $|P(u,v)| \leq A [K(u,v)]^{\frac{d}{4}}$ for some constant $A=A(c)>0$. Since $d\leq 3$, we apply the Young inequality to obtain for any $\epsilon>0$,
		\[
		A[K(u,v)]^{\frac{d}{4}} \leq \epsilon K(u,v) + C(\epsilon,A).
		\]
		Thus
		\begin{align} \label{bou-kin}
		E(u,v) \geq \left(\frac{1}{2}-\epsilon\right) K(u,v) - C(\epsilon,A).
		\end{align}
		By choosing $0<\epsilon<\frac{1}{2}$, we see that $E(u,v) >-C(\epsilon,A)$. This shows that $J(c)>-\infty$. The proof for $J(c)<0$ is similar to the one for $I(a,b)<0$ given in Lemma $\ref{lem-neg}$. The boundeness in $H^1 \times H^1$ of any minimizing sequence for $J(c)$ follows similarly as in the proof of Item (1) of Lemma $\ref{lem-pro-min}$. The proof is complete.
	\end{proof}

	\begin{lemma} \label{lem-spl-pro}
		Let $d\leq 3$ and $c>0$. If $(u_n,v_n)_{n\geq 1}$ is a minimizing problem for $J(c)$, then there exist a subsequence, still denoted by $(u_n,v_n)_{n\geq 1}$, and a number $0<a <c$ such that 
		\[
		\lim_{n\rightarrow \infty} \|u_n\|^2_{L^2} =a, \quad \lim_{n\rightarrow \infty} E(u_n,v_n) = I\left(a,\frac{c-a}{2}\right).
		\]
		In particular, $J(c) = I\left(a,\frac{c-a}{2}\right)$.
	\end{lemma}
	\begin{proof}
		Since $M(u_n,v_n) \rightarrow c$ as $n\rightarrow \infty$, it follows that the sequence $(\|u_n\|^2_{L^2})_{n\geq 1}$ is bounded. Thus up to a subsequence, we can assume that $\|u_n\|^2_{L^2} \rightarrow a$ as $n\rightarrow \infty$ with $0\leq a \leq c$. Hence, $\|v_n\|^2_{L^2} \rightarrow \frac{c-a}{2}$ as $n\rightarrow \infty$. 
		
		We first claim that $a>0$. Suppose that $a=0$. Since $\|u_n\|^2_{L^2} \rightarrow 0$, the H\"older inequality, Sobolev embedding and the fact $(u_n,v_n)_{n\geq 1}$ is bounded in $H^1 \times H^1$ imply that 
		\[
		|P(u_n,v_n)| \rightarrow 0 \text{ as } n\rightarrow \infty.
		\]
		Thus
		\[
		J(c) = \lim_{n\rightarrow \infty} E(u_n,v_n) = \lim_{n\rightarrow \infty} \frac{1}{2} K(u_n,v_n) \geq 0,
		\]
		which is a contradiction to the fact $J(c)<0$. Similarly, we show $a<c$. 
		
		Finally, we show $J(c) = I\left(a,\frac{c-a}{2}\right)$. Let $(u,v)$ be a minimizer of $I\left(a,\frac{c-a}{2}\right)$ (This is possible due to Theorem $\ref{theo-exi-min}$). It follows that $\|u\|^2_{L^2} =a$ and $\|v\|^2_{L^2}=\frac{c-a}{2}$, hence $M(u,v) = c$. Thus
		\begin{align} \label{equ-pro-1}
		J(c) \leq E(u,v) =I\left(a,\frac{c-a}{2}\right).
		\end{align}
		We now show the inverse inequality. Since $0<a<c$, we see that $a$ and $\frac{c-a}{2}$ are both positive. We set 
		\[
		\alpha_n = \frac{\sqrt{a}}{\|u_n\|_{L^2}}, \quad \beta_n = \frac{\sqrt{\frac{c-a}{2}}}{\|v_n\|_{L^2}}.
		\]
		It follows that
		\[
		\|\alpha_n u_n\|^2_{L^2} =a, \quad \|\beta_n v_n\|^2_{L^2} = \frac{c-a}{2}.
		\]
		Note that the scallings $\alpha_n$ and $\beta_n$ tend to 1 as $n\rightarrow \infty$. Thus
		\[
		I\left(a,\frac{c-a}{2}\right) \leq \lim_{n\rightarrow \infty} E(\alpha_n u_n, \beta_n v_n) = \lim_{n\rightarrow \infty} E(u_n,v_n) = J(c).
		\]
		This combined with \eqref{equ-pro-1} imply that $J(c) = I\left(a,\frac{c-a}{2}\right)$. The proof is complete.
	\end{proof}

	We are now able to prove Theorem $\ref{theo-exi-min-J}$.
	
	\noindent \textit{Proof of Theorem $\ref{theo-exi-min-J}$.}
		Let $(u_n,v_n)_{n\geq 1}$ be a minimizing sequence for $J(c)$. By Lemma $\ref{lem-spl-pro}$, we see that there exists $0<a<c$ such that up to a subsequence, $(u_n,v_n)_{n\geq 1}$ is a minimizing sequence for $I\left(a,\frac{c-a}{2}\right)$. Since both $a$ and $\frac{c-a}{2}$ are positive, by Item (1) of Theorem $\ref{theo-exi-min}$, there exist $(y_n)_{n\geq 1} \subset \mathbb{R}^d$ and $(u,v) \in \mathcal{G}_{a,\frac{c-a}{2}}$ such that $(u_n(\cdot+y_n), v_n(\cdot+y_n))_{n\geq 1}$ has a subsequence converging to $(u,v)$ in $H^1 \times H^1$. Since $\|u\|^2_{L^2} = a, \|v\|^2_{L^2} = \frac{c-a}{2}$ and $E(u,v) = I\left(a,\frac{c-a}{2}\right) = J(c)$, it follows that
		\[
		M(u,v) = c, \quad E(u,v) = J(c).
		\]
		This implies that $(u,v)$ is a minimizer for $J(c)$ and hence $(u,v) \in \mathcal{M}_c$.
		
		Let us now prove \eqref{lit-gro-J-1}. Assume by contradiction that \eqref{lit-gro-J-1} is not true, then there exist a subsequence $(u_{n_k},v_{n_k})_{k\geq 1}$ of $(u_n,v_n)$ and a constant $\delta>0$ such that
		\begin{align} \label{lit-gro-J-pro}
		\lim_{n\rightarrow \infty} \inf_{(w,z) \in \mathcal{M}_c, y\in \mathbb{R}^d} \|(u_{n_k}(\cdot+y), v_{n_k}(\cdot+y)) - (w,z) \|_{H^1\times H^1 } \geq \delta.
		\end{align}
		Since $(u_{n_k},v_{n_k})_{k\geq 1}$ is still a minimizing sequence of $J(c)$, by Lemma \ref{lem-spl-pro}, there exists a number $0<a<c$ such that up to subsequence, $(u_{n_k}, v_{n_k})_{k\geq 1}$ is a minimizing sequence for $I\left(a,\frac{c-a}{2}\right)$ and $J(c) = I\left(a,\frac{c-a}{2}\right)$. By Item (1) of Theorem $\ref{theo-exi-min}$, there exist a sequence $(y_k)_{k\geq 1} \subset \mathbb{R}^d$ and $(g,h) \in \mathcal{G}_{a,\frac{c-a}{2}}$ such that up to subsequence,
		\[
		\lim_{n\rightarrow \infty} \|(u_{n_k}(\cdot+y_k), v_{n_k}(\cdot+y_k))- (g,h)\|_{H^1 \times H^1} =0.
		\]
		Now note that $(g,h) \in \mathcal{G}_{a,\frac{c-a}{2}}$  and $J(c) = I\left(a,\frac{c-a}{2}\right)$ imply that $(g,h) \in \mathcal{M}_c$. We thus get a contradiction to \eqref{lit-gro-J-pro}. 
		
		Item (3) follows from Item (2) since if $(w,z) \in \mathcal{M}_c$, then $(w(\cdot-y), z(\cdot-y)) \in \mathcal{M}_c$.
		
		We next prove Item (4). Since $(u,v)$ is a minimizer of $J(c)$, there exists a Lagrange multiplier $\omega \in \mathbb{R}$ such that
		\[
		H'(u,v)[\chi,\vartheta] =0, \quad \forall (\chi,\vartheta) \in C^\infty_0 \times C^\infty_0,
		\]
		where $H(u,v) = E(u,v) + \frac{\omega}{2} M(u,v)$. By the same calculations as in the proof of Theorem $\ref{theo-exi-min}$, we see that $(u,v)$ is a classical solution of \eqref{sys-ell-equ} with $\omega_2=2\omega_1=2\omega \in \mathbb{R}$. Before finishing the proof, we need the following result.
		\begin{lemma} \label{lem-pro-sol}
			Let $(\phi,\psi) \in H^1 \times H^1$ be a solution to \eqref{sys-ell-equ} with $\omega_2=2\omega_1=2\omega$. Then the following identities hold:
			\begin{align*}
			K(\phi,\psi) + \omega M(\phi,\psi) &= 3 P(\phi,\psi), & P(\phi,\psi) &=2 S_\omega(\phi,\psi), \\
			K(\phi,\psi) &= d S_\omega(\phi,\psi), & \omega M(\phi,\psi) &= (6-d) S_\omega(\phi,\psi).
			\end{align*}
		\end{lemma}
		\begin{proof}
			Multiplying both sides of the first equation in \eqref{sys-ell-equ} with $\overline{\phi}$, integrating over $\mathbb{R}^d$ and taking the real part, we have
			\[
			\|\nabla \phi\|^2_{L^2} + \omega \|\phi\|_{L^2}^2 = 2 \text{Re} \int \overline{\psi} \phi^2 dx.
			\]
			Similarly, multiplying both sides of the second equation in \eqref{sys-ell-equ} with $\overline{\psi}$, integrating over $\mathbb{R}^d$ and taking the real part, we get
			\[
			\frac{1}{2} \|\nabla \psi\|^2_{L^2} + 2\omega \|\psi\|^2_{L^2} = \text{Re} \int \overline{\psi} \phi^2 dx.
			\]
			Adding these two equalities, we obtain $K(\phi,\psi) + \omega M(\phi,\psi) = 3 P(\phi,\psi)$. Using this identity together with the fact $S_\omega(\phi,\psi) = \frac{1}{2}(K(\phi,\psi) + \omega M(\phi,\psi)) - P(\phi,\psi)$, it yields that $P(\phi,\psi) = 2 S_\omega(\phi,\psi)$. Multiplying both sides of the first equation with $x\cdot \nabla \overline{\phi}$, integrating over $\mathbb{R}^d$ and taking the real part, we have
			\[
			-\text{Re} \int \Delta \phi x \cdot \nabla \overline{\phi} dx + \omega \text{Re} \int \phi x \cdot \nabla \overline{\phi} dx = 2 \text{Re} \int \psi \overline{\phi} x \cdot \nabla \overline{\phi} dx.
			\]
			A direct calculation shows that
			\begin{align*}
			\text{Re} \int \Delta \phi x \cdot \nabla \overline{\phi} dx &= \frac{d-2}{2} \|\nabla \phi\|^2_{L^2}, \\
			\text{Re} \int \phi x \cdot \nabla \overline{\phi} dx &= -\frac{d}{2} \|\phi\|^2_{L^2}, \\
			\text{Re} \int \psi \overline{\phi} x \cdot \nabla \overline{\phi} dx &= - \frac{d}{2} \text{Re} \int \overline{\psi} \phi^2 dx - \frac{1}{2} \text{Re} \int \phi^2 x \cdot \nabla \overline{\psi} dx.
			\end{align*}
			We thus get
			\begin{align} \label{poh-ide-1}
			-\frac{d-2}{2} \|\nabla \phi\|^2_{L^2} -\frac{d\omega}{2} \|\phi\|^2_{L^2} = -d \text{Re} \int \overline{\psi} \phi^2 dx - \text{Re} \int \phi^2 x \cdot \nabla \overline{\psi} dx.
			\end{align}
			By the same argument, multiplying both sides of the second equation in \eqref{sys-ell-equ} with $x \cdot \nabla \overline{\psi}$, integrating over $\mathbb{R}^d$ and taking the real part, we get
			\begin{align} \label{poh-ide-2}
			-\frac{d-2}{4} \|\nabla \psi\|^2_{L^2} - d \omega \|\psi\|^2_{L^2} = \text{Re} \int \phi^2 x \cdot \nabla \overline{\psi} dx.
			\end{align}
			Adding \eqref{poh-ide-1} and \eqref{poh-ide-2} together, we obtain
			\[
			\frac{d-2}{2} K(\phi,\psi) + \frac{d\omega}{2} M(\phi,\psi) = d P(\phi,\psi).
			\]
			This identity together with the definition of $S_\omega(\phi,\psi)$ imply that $K(\phi,\psi) = d S_\omega(\phi,\psi)$. The last identity also yields that $\omega M(\phi,\psi) = (6-d) S_\omega(\phi,\psi)$. The proof is complete.
		\end{proof}
		We now continue the proof of Item (4) by showing that $\omega>0$. Since $(u,v)$ is a solution of \eqref{sys-ell-equ} with $\omega_2=2\omega_1=2\omega$, we have the following identities
		\[
		K(u,v) = d S_\omega(u,v), \quad \omega M(u,v)=(6-d) S_\omega(u,v).
		\]
		This implies $\omega>0$ since $d\leq 3$ and $K(u,v), M(u,v), S_\omega(u,v)$ are all positive. The rest of Item (4) follows exactly the same as in the proof of Item (4) in Theorem $\ref{theo-exi-min}$. 
		
		Finally, we show that each $(u,v) \in \mathcal{M}_c$ is a ground state for \eqref{sys-ell-equ} with $\omega_2=2\omega_1=2\omega>0$. Let $(\tilde{u},\tilde{v})$ be a solution of \eqref{sys-ell-equ} with $\omega_2=2\omega_1=2\omega>0$. Our goal is to show
		\[
		S_\omega(u,v) \leq S_\omega(\tilde{u},\tilde{v}).
		\]
		Assume by contradiction that $S_\omega(\tilde{u},\tilde{v})< S_\omega(u,v)$. Since $(\tilde{u},\tilde{v})$ is a solution of \eqref{sys-ell-equ} with $\omega_2=2\omega_1=2\omega>0$, we have the following identities
		\[
		K(\tilde{u},\tilde{v})=d S_\omega(\tilde{u},\tilde{v}), \quad \omega M(\tilde{u},\tilde{v})=(6-d) S_\omega(\tilde{u},\tilde{v}), \quad P(\tilde{u},\tilde{v})= 2 S_\omega(\tilde{u},\tilde{v}).
		\]
		In particular, we have
		\[
		E(\tilde{u},\tilde{v}) = \left(\frac{d}{2}-2\right) S_\omega(\tilde{u},\tilde{v}) = \left(\frac{d}{2}-2\right) \frac{\omega}{6-d} M(\tilde{u},\tilde{v}).
		\]
		Of course, similar identities hold for $(u,v)$ since $(u,v)$ is also a solution of \eqref{sys-ell-equ} with $\omega_2=2\omega_1=2\omega>0$. Now set
		\[
		\gamma := \left(\frac{c}{M(\tilde{u},\tilde{v})}\right)^{\frac{1}{4-d}}
		\]
		and define
		\[
		w(x) := \gamma^2 \tilde{u}(\gamma x), \quad z(x):= \gamma^2 \tilde{v}(\gamma x).
		\]
		It is easy to check that
		\[
		M(w,z)= \gamma^{4-d} M(\tilde{u},\tilde{v})=c.
		\]
		Since $(u,v) \in \mathcal{M}_c$, we have
		\begin{align*}
		\left(\frac{d}{2}-2\right) \frac{\omega}{6-d} c = E(u,v) \leq E(w,z) = \gamma^{6-d} E(\tilde{u},\tilde{v})= \gamma^{6-d} \left(\frac{d}{2}-2\right) \frac{\omega}{6-d} M(\tilde{u},\tilde{v}).
		\end{align*}
		This implies that $c\geq \gamma^{6-d} M(\tilde{u},\tilde{v})$ since $d\leq 3$. We thus get
		\[
		\gamma^{6-d} \leq \frac{c}{M(\tilde{u},\tilde{v})} = \gamma^{4-d} \text{ or } \gamma \leq 1.
		\]
		On the other hand, 
		\[
		\frac{\omega}{6-d} M(\tilde{u},\tilde{v}) = S_\omega(\tilde{u},\tilde{v})< S_\omega(u,v) = \frac{\omega}{6-d} c.
		\]
		Hence $M(\tilde{u},\tilde{v}) <c$ or $\gamma>1$ which is absurd. The proof of Item (5) is now complete.
	\hfill $\Box$

	We now prove the orbital stability of standing waves for \eqref{wor-Syst} given in Theorem $\ref{theo-sta}$. 
	
	\noindent \textit{Proof of Theorem $\ref{theo-sta}$.}
			Assume by contradiction that the claim is not true. Then there exists $\epsilon_0>0$ and a sequence of initial data $(u_{0,n}, v_{0,n})_{n\geq 1}$ such that
			\begin{align} \label{sta-pro-1}
			\inf_{(w,z) \in \mathcal{M}_c} \|(u_{0,n},v_{0,n})-(w,z)\|_{H^1 \times H^1} < \frac{1}{n},
			\end{align}
			and there exists a time sequence $(t_n)_{n\geq 1} \subset [0,+\infty)$ such that the corresponding solution sequence $(u_n(t_n), v_n(t_n))_{n\geq 1}$ of \eqref{wor-Syst} satisfies
			\begin{align} \label{sta-pro-2}
			\inf_{(w,z) \in \mathcal{M}_c} \|(u_n(t_n), v_n(t_n))-(w,z)\|_{H^1\times H^1} \geq \epsilon_0.
			\end{align}
			By \eqref{sta-pro-1} and the conservation of mass, we have
			\[
			M(u_n(t_n), v_n(t_n)) = M(u_{0,n}, v_{0,n}) \rightarrow M(w,z)=c \text{ as } n\rightarrow \infty.
			\]
			On the other hand, since $(u_{0,n},v_{0,n})_{n\geq 1} \rightarrow (w,z)$ in $H^1 \times H^1$ as $n\rightarrow \infty$, we have
			\begin{align*}
			\lim_{n\rightarrow \infty} E(u_{0,n}, v_{0,n}) &= \lim_{n\rightarrow \infty} \frac{1}{2} K(u_{0,n},v_{0,n}) - \text{Re} (\langle v_{0,n}, (u_{0,n})^2 \rangle) \\
			&= \frac{1}{2} K(w,z) - \text{Re} (\langle z, w^2 \rangle) = E(w,z).
			\end{align*}
			By the conservation of energy, we get
			\[
			E(u_n(t_n), v_n(t_n)) = E(u_{0,n}, v_{0,n}) \rightarrow E(w,z) = J(c).
			\]
			Thus, $(u_n(t_n), v_n(t_n))_{n\geq 1}$ is a minimizing sequence for $J(c)$. By Item (1) of Theorem $\ref{theo-exi-min-J}$, there exist a sequence $(y_n)_{n\geq 1} \subset \mathbb{R}^d$ and $(g,h) \in \mathcal{M}_c$ such that
			\[
			\lim_{n\rightarrow\infty} \|(u_n(t_n, \cdot+y_n), v_n(t_n, \cdot+y_n))-(g,h)\|_{H^1\times H^1} =0.
			\]
			Since $(g(\cdot-y_n), h(\cdot-y_n)) \in \mathcal{M}_c$, we see that for sufficiently large $n$,
			\[
			\inf_{(w,z) \in \mathcal{M}_c} \|(u_n(t_n),v_n(t_n))-(w,z)\|_{H^1 \times H^1} < \epsilon_0,
			\]
			which contradicts to \eqref{sta-pro-2}. The proof is complete.
		\hfill $\Box$

	\section{Existence and characterization of blow-up solutions}
	\label{S:2}
	\subsection{Existence of blow-up solutions}
	\label{Sub:2.1}
	In this subsection, we study the existence of blow-up solutions to \eqref{mas-res-Syst} with $d=4$ and $\kappa=\frac{1}{2}$. In \cite{HOT}, Hayashi-Ozawa-Takana showed the existence of finite time blow-up solutions in the case $E(u_0,v_0)<0$ and $(u_0, v_0) \in \Sigma \times \Sigma$, where $\Sigma = H^1 \cap L^2(|x|^2 dx)$. Moreover, they pointed out explicit solutions which blows up in finite time and has mass equal to the mass of the ground state. Note that in this case the enery is zero. Our goal is to prove that for initial data in $H^1 \times H^1$ (not necessary finite variance or radially symmetric) with negative energy, then the corresponding solution either blows up in finite time or blows up infinite time. Moreover, one can rule out the infinite time blow-up by considering additionally initial data has finite variance (i.e. $(xu_0,xv_0) \in L^2 \times L^2$) or is radially symmetric.
	
	To do so, we need some virial estimates related to \eqref{mas-res-Syst}. 
	Given a real-valued function $\chi$, we define the virial potential $V_\chi$ associated to \eqref{mas-res-Syst} by
	\[
	V_\chi(t):= \int \chi(x) (|u(t,x)|^2 + 2 |v(t,x)|^2) dx.
	\]
	We have the following virial identity related to \eqref{mas-res-Syst}. 
	\begin{lemma} \label{lem-vir-ide-chi}
		Let $d=4$ and $\kappa=\frac{1}{2}$. Let $(u(t),v(t))$ be a solution to \eqref{mas-res-Syst}. Then it holds that
		\begin{align} \label{vir-ide-chi-1}
		\frac{d}{dt} V_\chi(t) = 2 \int \nabla \chi \cdot \emph{Im}( \nabla u(t) \overline{u}(t) + \nabla v(t) \overline{v}(t) ) dx,
		\end{align}
		and 
		\begin{align} \label{vir-ide-chi-2}
		\frac{d^2}{dt^2} V_\chi(t) = -\int \Delta^2 \chi \left( |u(t)|^2 +\frac{1}{2} |v(t)|^2 \right) dx &+ \sum_{jk} \emph{Re} \int \partial^2_{jk} \chi \left(4 \partial_j u(t) \partial_k \overline{u}(t) + 2 \partial_j v(t) \partial_k \overline{v}(t) \right) dx \nonumber \\
		&- 2 \emph{Re} \int \Delta \chi \overline{v} (t) u^2(t) dx. 
		\end{align}
	\end{lemma}
	
	\begin{proof}
		We only make a formal calculation. The rigorous proof requires a regularization procedure. Since $(u(t),v(t))$ is a solution to \eqref{mas-res-Syst}, we have
		\begin{align*}
		\frac{d}{dt} V_\chi(t) &= 2 \int \chi \text{Re}(\partial_t u \overline{u} + 2 \partial_t v \overline{v}) dx \\
		&= -2 \int \chi \text{Im}(\Delta u +2 v \overline{u}) \overline{u} + (\Delta v + 2 u^2) \overline{v}) dx \\
		&= -2 \int \chi \text{Im}(\Delta u \overline{u} + \Delta v \overline{v}) dx \\
		&= 2 \int \nabla \chi \cdot \text{Im}(\nabla u \overline{u} + \nabla v \overline{v}) dx.
		\end{align*}
		We next have
		\begin{align*}
		\frac{d^2}{dt^2} V_\chi(t) &= 2 \int \nabla \chi \cdot  \text{Im} \left( \nabla \partial_t u \overline{u} + \nabla u \partial_t \overline{u} + \nabla \partial_t v  \overline{v} + \nabla v \partial_t \overline{v} \right) dx \\
		&= 2 \int \nabla \chi \cdot \text{Re} \left( \nabla (\Delta u + 2 v \overline{u}) \overline{u} - \nabla u (\Delta \overline{u} + 2 \overline{v} u) \right) dx \\
		&\mathrel{\phantom{=}} + \int \nabla \chi \cdot \text{Re} \left( \nabla (\Delta v + 2 u^2) \overline{v} - \nabla v(\Delta \overline{v} + 2(\overline{u})^2)\right) dx.
		\end{align*}
		A direct computation shows that
		\begin{align*}
		\int \nabla \chi \cdot \text{Re} (\nabla (\Delta u + 2 v \overline{u}) \overline{u}) dx &= \sum_{j} \text{Re} \int \partial_j \chi \partial_j (\Delta u + 2 v \overline{u}) \overline{u} dx \\
		&= -\sum_{j} \text{Re} \int \partial^2_j \chi(\Delta u + 2 v \overline{u}) \overline{u} + \partial_j \chi \partial_j \overline{u} (\Delta u + 2 v \overline{u})  dx \\
		&= - \text{Re} \int \Delta \chi \Delta u \overline{u} +2 \Delta \chi \overline{v} u^2 + \nabla \chi \cdot \nabla u (\Delta \overline{u} + 2 \overline{v} u) dx.
		\end{align*}
		We also have
		\begin{align*}
		\text{Re} \int \Delta \chi \Delta u \overline{u} dx &= \sum_{jk} \text{Re} \int \partial^2_j \chi \partial^2_k u \overline{u} dx \\
		&= -\sum_{jk} \text{Re} \int \partial^3_{jjk} \chi \partial_k u \overline{u} + \partial^2_j \chi \partial_k u \partial_k \overline{u} dx \\
		&= \sum_{jk} \text{Re} \int \partial^4_{jjkk} \chi |u|^2 + \partial^3_{jjk} \chi u \partial_k \overline{u} dx - \int \Delta \chi |\nabla u|^2 dx \\
		&= \frac{1}{2} \int \Delta^2 \chi |u|^2 dx - \int \Delta \chi |\nabla u|^2 dx.
		\end{align*}
		Thus
		\begin{align*}
		\int \nabla \chi \cdot \text{Re}( \nabla(\Delta u + 2 v \overline{u}) \overline{u}) dx = -\frac{1}{2} \int \Delta^2 \chi |u|^2 + \int \Delta \chi |\nabla u|^2 dx &- 2 \text{Re} \int \Delta \chi \overline{v} u^2 dx \\
		&- \text{Re} \int \nabla \chi \cdot \nabla u (\Delta \overline{u} + 2 \overline{v} u) dx.
		\end{align*}
		Similarly,
		\begin{align*}
		\int \nabla \chi \cdot \text{Re}(\nabla (\Delta v + 2 u^2) \overline{v}) dx = -\frac{1}{2} \int \Delta^2 \chi |v|^2 dx + \int \Delta \chi |\nabla v|^2 &- 2 \text{Re} \int \Delta \chi \overline{v} u^2 dx \\
		&- \text{Re} \int \nabla \chi \cdot \nabla v (\Delta \overline{v} + 2 (\overline{u})^2) dx.
		\end{align*}
		We get
		\begin{align} \label{vir-ide-chi-pro-1}
		\frac{d^2}{dt^2} V_\chi(t)&= - \int \Delta^2 \chi \left(|u|^2 + \frac{1}{2} |v|^2\right) dx + 2 \int \Delta \chi \left( |\nabla u|^2 + \frac{1}{2} |\nabla v|^2 \right) dx - 6 \text{Re} \int \Delta \chi \overline{v} u^2 dx \nonumber\\
		&\mathrel{\phantom{=}}- 4 \text{Re} \int \nabla \chi \cdot \nabla u (\Delta \overline{u} + 2 \overline{v} u) dx - 2 \text{Re} \int \nabla \chi \cdot \nabla v ( \Delta \overline{v} + 2 (\overline{u})^2 ) dx.
		\end{align}
		Now
		\begin{align} \label{vir-ide-chi-pro-2}
		\text{Re} \int \nabla \chi \cdot \nabla u (\Delta \overline{u} + 2 \overline{v} u) dx &= \text{Re} \int \nabla \chi \cdot \nabla u \Delta \overline{u} dx + 2 \text{Re} \int \nabla \chi \cdot \nabla u \overline{v} u dx \nonumber \\
		&= -\sum_{jk} \text{Re} \int \partial^2_{jk} \chi \partial_j u \partial_k \overline{u} dx + \frac{1}{2}\int \Delta \chi |\nabla u|^2 dx +2 \text{Re}\int \nabla \chi \cdot \nabla u \overline{v} u dx.
		\end{align}
		Similarly,
		\begin{align} \label{vir-ide-chi-pro-3}
		\text{Re} \int \nabla \chi \cdot \nabla v ( \Delta \overline{v} + 2 (\overline{u})^2) dx = - \sum_{jk} \text{Re} \int \partial^2_{jk}\chi \partial_j v \partial_k \overline{v} dx &+ \frac{1}{2} \int \Delta \chi |\nabla v|^2 dx - 2 \text{Re}\int \Delta \chi \overline{v} u^2 dx \nonumber\\
		&- 4 \text{Re} \int \nabla \chi \cdot \nabla u \overline{v} u dx.
		\end{align}
		Combining \eqref{vir-ide-chi-pro-1}, \eqref{vir-ide-chi-pro-2} and \eqref{vir-ide-chi-pro-3}, we prove the result.
	\end{proof}
	
	In order to prove the existence of finite time blow-up solutions for \eqref{mas-res-Syst} with radial data, we need localized virial estimates. To do so, we introduce a function $\theta: [0,\infty) \rightarrow [0,\infty)$ satisfying
	\begin{align} \label{def-the}
	\theta(r) = \left\{
	\begin{array}{cl}
	r^2 & \text{if } 0 \leq r \leq 1, \\
	\text{2} &\text{if } r \geq 2, 
	\end{array}
	\right. \quad \text{and} \quad \theta''(r) \leq 2 \text{ for } r \geq 0. 
	\end{align}
	For $R>0$, we define the radial function
	\begin{align} \label{def-chi-R}
	\chi_R(x) = \chi_R(r) := R^2 \theta(r/R), \quad r=|x|.
	\end{align}
	We see easily that
	\begin{align} \label{pro-chi-R}
	2- \chi''_R(r) \geq 0, \quad 2 - \frac{\chi'_R(r)}{r} \geq 0, \quad 8 - \Delta \chi_R(x) \geq 0.
	\end{align}
	\begin{lemma} \label{lem-loc-vir-est}
		Let $d=4$ and $\kappa=\frac{1}{2}$. Let $R>0$ and $\chi_R$ be as in \eqref{def-chi-R}. Let $(u(t),v(t))$ be a radial solution to \eqref{mas-res-Syst} defined on the maximal time interval $[0,T)$. Then for any $\epsilon>0$ and any $t \in [0,T)$,
		\begin{align} \label{loc-vir-est} 
		\frac{d^2}{dt^2} V_{\chi_R}(t) \leq 16 E(u(t),v(t)) &- 4 \int (\chi_{1,R} - C\epsilon \chi_{2,R}^2) |\nabla u(t)|^2 dx \nonumber \\
		& -2 \int (\chi_{1,R} - C\epsilon \chi_{2,R}^2) |\nabla v(t)|^2 dx + O(R^{-2} + \epsilon R^{-2} + \epsilon^{-1/3} R^{-2}),
		\end{align}
		for some constant $C>0$, where 
		\begin{align}  \label{def-chi-12-R}
		\chi_{1,R}= 2-\chi''_R, \quad \chi_{2,R} = 8-\Delta \chi_R.
		\end{align}
		The implicit constant depends only on the conserved mass $M(u,v)$.
	\end{lemma}
	
	\begin{proof}
		Thanks to the virial identity \eqref{vir-ide-chi-2}, we write
		\begin{align*}
		\frac{d^2}{dt^2} V_{\chi_R}(t) = 16 E(u,v) &-\int \Delta^2 \chi_R \left(|u|^2 + \frac{1}{2}|v|^2 \right) dx +\sum_{jk} \text{Re} \int \partial^2_{jk} \chi_{R} ( 4 \partial_j u \partial_k \overline{u} + 2 \partial_j v \partial_k \overline{v}) dx \\
		& - 2 \text{Re} \int \Delta \chi_R \overline{v} u^2 dx - 8\left( \|\nabla u\|^2_{L^2} + \frac{1}{2} \|\nabla v\|^2_{L^2} \right) + 16 \text{Re} \int \overline{v} u^2 dx.
		\end{align*}
		Since $(u,v)$ is radial, we use the fact
		\begin{align} \label{rad-for}
		\partial_j = \frac{x_j}{r} \partial_r, \quad \partial^2_{jk} = \left( \frac{\delta_{jk}}{r} - \frac{x_j x_k}{r^3} \right) \partial_r + \frac{x_j x_k}{r^2} \partial^2_r
		\end{align}
		to have
		\[
		\sum_{jk} \partial^2_{jk} \chi_R \partial_j u \partial_k \overline{u} = \chi''_R |\partial_r u|^2 = \chi''_R |\nabla u|^2.
		\]
		We thus write
		\begin{align*}
		\frac{d^2}{dt^2} V_{\chi_R}(t) = 16 E(u,v) - \int \Delta^2 \chi_R \left( |u|^2 + \frac{1}{2} |v|^2 \right) dx  - 4 \int \chi_{1,R} |\nabla u|^2 dx &- 2 \int \chi_{1,R} |\nabla v|^2 dx \\
		&+ 2 \text{Re} \int \chi_{2,R} \overline{v} u^2 dx,
		\end{align*}
		where $\chi_{1,R}$ and $\chi_{2,R}$ are as in \eqref{def-chi-12-R}. We next use the radial Sobolev embedding (see e.g. \cite{Strauss})
		\[
		\sup_{x \ne 0} |x|^{\frac{d-1}{2}} |f(x)| \leq C(d) \|\nabla f\|_{L^2}^{\frac{1}{2}} \|f\|^{\frac{1}{2}}_{L^2}
		\]
		and the conservation of mass to estimate
		\begin{align*}
		\left| \text{Re} \int \chi_{2,R} \overline{v} u^2 dx \right| & \leq \int_{|x|>R} \chi_{2,R} |v| |u|^2 dx \\
		&\leq \left( \sup_{|x|>R} |\chi_{2,R}(x) v(t,x)| \right) \|u\|^2_{L^2} \\
		&\leq R^{-\frac{3}{2}} \left( \sup_{|x|>R} |x|^{\frac{3}{2}} |\chi_{2,R}(x) v(t,x)| \right) \|u\|^2_{L^2} \\
		&\lesssim R^{-\frac{3}{2}} \|\nabla (\chi_{2,R} v)\|^{\frac{1}{2}}_{L^2} \|\chi_{2,R} v\|^{\frac{1}{2}}_{L^2} \|u\|^2_{L^2} \\
		&\lesssim R^{-\frac{3}{2}} \|\nabla (\chi_{2,R} v)\|^{\frac{1}{2}}_{L^2}.
		\end{align*}
		We next use the Young inequality to have for any $\epsilon>0$,
		\begin{align*}
		R^{-\frac{3}{2}} \|\nabla(\chi_{2,R} v)\|^{\frac{1}{2}}_{L^2} &\lesssim \epsilon \|\nabla (\chi_{2,R} v)\|^2_{L^2} +\epsilon^{-\frac{1}{3}} R^{-2} \\
		& \leq \epsilon \left(\|(\nabla \chi_{2,R}) v\|^2_{L^2} + \|\chi_{2,R} \nabla v\|^2_{L^2}\right) + \epsilon^{-\frac{1}{3}} R^{-2} \\
		&\lesssim \epsilon \|\chi_{2,R} \nabla v\|^2_{L^2} + \epsilon R^{-2} + \epsilon^{-\frac{1}{3}} R^{-2}.
		\end{align*}
		Here we use the fact that $\|\nabla \chi_{2,R}\|_{L^\infty} \lesssim R^{-2}$. Therefore,
		\begin{align} \label{vir-ide-chi-2-pro-1}
		\left| \text{Re}\int \chi_{2,R} \overline{v} u^2 dx \right| \lesssim \epsilon \int \chi_{2,R}^2 |\nabla v|^2 dx + \epsilon R^{-2} + \epsilon^{-\frac{1}{3}} R^{-2}.
		\end{align}
		Similarly, we have
		\[
		\left| \text{Re} \int \chi_{2,R} \overline{v} u^2 dx \right| \leq \int_{|x|>R} \chi_{2,R} |v| |u|^2 dx \leq \left(\sup_{|x|>R} |\chi_{2,R}(x) u(t,x)| \right) \|v\|_{L^2} \|u\|_{L^2}.
		\]
		Thus,
		\begin{align} \label{vir-iden-chi-2-pro-2}
		\left| \text{Re} \int \chi_{2,R} \overline{v} u^2 dx \right| \lesssim \epsilon \int \chi^2_{2,R} |\nabla u|^2 dx + \epsilon R^{-2} + \epsilon^{-\frac{1}{3}} R^{-2}.
		\end{align}
		Combining \eqref{vir-ide-chi-2-pro-1} and \eqref{vir-iden-chi-2-pro-2}, we prove the claim.
	\end{proof}
	
	We first prove the blow-up criteria for \eqref{mas-res-Syst} given in Theorem $\ref{theo-blo-non-rad}$.
	
	\noindent \textit{Proof of Theorem $\ref{theo-blo-non-rad}$.}
		The proof is based on the argument of Du-Wu-Zhang \cite{DWZ}. If $T<+\infty$, we are done. If $T=+\infty$, then we need to show \eqref{blo-non-rad}. Assume by contradiction that the solution exists globally in time and satisfies
		\begin{align} \label{blo-non-rad-pro-1}
		\sup_{t\in [0,+\infty)} \|(u(t),v(t))\|_{H^1 \times H^1} <\infty. 
		\end{align}
		The key step is to control $L^2$-norm of the solution outside a large ball. To do this, we introduce $\vartheta: [0,\infty) \rightarrow [0,1]$ a smooth function satisfying $\vartheta(r) =0$ if $0\leq r \leq 1/2$ and $\vartheta(r) = 1$ if $r \geq 1$. Given $R>0$, we define the radial function
		\[
		\varphi_R(x) = \varphi_R(r):= \vartheta(r/R), \quad r=|x|.
		\]
		It is easy to see that $\|\nabla \varphi_R\|_{L^\infty} \lesssim R^{-1}$. We next define 
		\[
		I_{\varphi_R} (t) := \int \varphi_R(x) \left( |u(t,x)|^2 + 2 |v(t,x)|^2 \right) dx.
		\]
		By the fundamental theorem of calculus, we have
		\[
		I_{\varphi_R}(t) = I_{\varphi_R}(0) + \int_0^t \frac{d}{ds} I_{\varphi_R}(s) ds \leq V_{\varphi_R}(0) + \left(\sup_{s \in [0,t]} \left|\frac{d}{ds} I_{\varphi_R}(s)\right| \right) t.
		\]
		By \eqref{blo-non-rad-pro-1} and the conservation of mass, 
		\begin{align*}
		\sup_{s \in [0,t]} \left| \frac{d}{ds} I_{\varphi_R}(s) \right| &= \sup_{s \in [0,t]} \left| 2 \int \nabla \varphi_R \cdot \text{Im}(\nabla u \overline{u} + \nabla v \overline{v}) dx \right| \\
		&\lesssim  \|\nabla \varphi_R\|_{L^\infty} \sup_{s\in [0,t]} (\|\nabla u(s)\|_{L^2} \|u(s)\|_{L^2} + \|\nabla v(s)\|_{L^2} \|v(s)\|_{L^2}) \\
		&\lesssim \|\nabla \varphi_R\|_{L^\infty} \sup_{s\in [0,t]} \|(u(s), v(s))\|_{H^1 \times H^1} \\
		&\lesssim CR^{-1},
		\end{align*}
		for some constant $C>0$ independent of $R$. We thus obtain
		\[
		I_{\varphi_R}(t) \leq V_{\varphi_R}(0) + CR^{-1} t. 
		\]
		It follows from the choice of $\vartheta$ and the conservation of mass that
		\[
		I_{\varphi_R}(0) = \int \varphi_R(x) (|u_0(x)|^2 + 2 |v_0(x)|^2) dx \leq \int_{|x|>R/2} |u_0(x)|^2 + 2 |v_0(x)|^2 dx \rightarrow 0 \text{ as } R\rightarrow \infty.
		\]
		Thus $I_{\varphi_R}(0) = o_R(1)$. Since 
		\[
		\int_{|x|\geq R} |u(t,x)|^2 + 2 |v(t,x)|^2 dx \leq I_{\varphi_R}(t),
		\]
		we obtain the following control on the $L^2$-norm of the solution outside a large ball.
		\begin{lemma} \label{lem-con-L2-nor}
			Let $d=4$ and $\kappa=1$. Let $\epsilon>0$ and $R>0$. Then there exists a constant $C$ independent of $R$ such that for any $t\in [0,T_0]$ with $T_0:= \frac{\epsilon R}{C}$, 
			\begin{align} \label{con-L2-nor}
			\int_{|x|\geq R} |u(t,x)|^2 + 2 |v(t,x)|^2 \leq o_R(1) + \epsilon.
			\end{align}
		\end{lemma}
		Now let $\chi_R$ be as in \eqref{def-chi-R}. By \eqref{vir-ide-chi-2}, 
		\begin{align*}
		\frac{d^2}{dt^2} V_{\chi_R}(t) = -\int \Delta^2 \chi_R \left(|u|^2 +\frac{1}{2}|v|^2\right) dx &+ \sum_{jk} \text{Re} \int \partial^2_{jk} \chi_R (4 \partial_j u \partial_k \overline{u} + 2 \partial_j u \partial_k \overline{u}) dx \\
		&- 2\text{Re} \int \Delta \chi_R \overline{v} u^2 dx.
		\end{align*}
		Since $\chi_R$ is radial, by using \eqref{rad-for}, it is not hard to check that
		\[
		\sum_{jk} \text{Re} \int \partial^2_{jk}\chi_R \partial_j u \partial_k \overline{u} dx = \int \frac{\chi'_R}{r} |\nabla u|^2 dx + \int \left(\frac{\chi''_R}{r^2}-\frac{\chi'_R}{r^3}\right) |x\cdot \nabla u|^2 dx,
		\]
		and similarly for $v$. We thus write
		\begin{align*}
		\frac{d^2}{dt^2} V_{\chi_R}(t) &= 16E(u,v)- 8 \left( \|\nabla u\|^2_{L^2}+\frac{1}{2}\|\nabla v\|^2_{L^2}\right) + 16 \text{Re} \int \overline{v} u^2 dx \\
		&\mathrel{\phantom{=}} -\int \Delta^2 \chi_R\left(|u|^2 +\frac{1}{2} |v|^2 \right) dx  - 2 \text{Re} \int \Delta \chi_R \overline{v} u^2 dx \\
		&\mathrel{\phantom{=}} + 4 \int \frac{\chi'_R}{r} |\nabla u|^2 dx + 4 \int \left(\frac{\chi''_R}{r^2} - \frac{\chi'_R}{r^3} \right) |x \cdot \nabla u|^2 dx \\
		&\mathrel{\phantom{=}} + 2 \int \frac{\chi'_R}{r} |\nabla v|^2 dx + 2 \int \left(\frac{\chi''_R}{r^2} - \frac{\chi'_R}{r^3} \right) |x \cdot \nabla v|^2 dx \\
		&= 16 E(u,v) - \int \Delta^2 \chi_R \left( |u|^2 +\frac{1}{2} |v|^2 \right) dx + 2 \text{Re} \int (8-\Delta \chi_R) \overline{v} u^2 dx \\
		&\mathrel{\phantom{=}} + 4 \int \left(\frac{\chi'_R}{r}-2\right) |\nabla u|^2 dx + 4 \int \left(\frac{\chi''_R}{r^2} - \frac{\chi'_R}{r^3}\right) |x \cdot \nabla u|^2 dx \\
		&\mathrel{\phantom{=}} + 2 \int \left(\frac{\chi'_R}{r}-2\right) |\nabla v|^2 dx + 2 \int \left(\frac{\chi''_R}{r^2} - \frac{\chi'_R}{r^3}\right) |x \cdot \nabla v|^2 dx.
		\end{align*}
		Since $\|\Delta^2 \chi_R\|_{L^\infty} \lesssim R^{-2}$ and $\text{supp}(\Delta^2 \chi_R) \subset \{ |x| \geq R\}$, we have
		\[
		\int \Delta^2 \chi_R \left(|u|^2 +\frac{1}{2}|v|^2 \right) dx \lesssim R^{-2} \left(\|u\|^2_{L^2(|x| \geq R)} + 2 \|v\|^2_{L^2(|x|\geq R)}\right).
		\]
		Since $\chi''_R \leq 2$, the Cauchy-Schwarz inequality $|x\cdot \nabla u| \leq |x| |\nabla u| = r |\nabla u$ implies that
		\[
		\int \left(\frac{\chi'_R}{r}-2\right) |\nabla u|^2 dx + \int \left(\frac{\chi''_R}{r^2}-\frac{\chi'_R}{r^3}\right) |x \cdot \nabla u|^2 dx \leq 0,
		\]
		and similarly for $v$. Moreover, since $\|8-\Delta \chi_R\|_{L^\infty} \lesssim 1$ and $\text{supp}(8-\Delta \chi_R) \subset \{ |x| \geq R\}$, the Sobolev embedding and \eqref{blo-non-rad-pro-1} imply that
		\begin{align*}
		\left| \text{Re} \int (8-\Delta \chi_R) \overline{v} u^2 dx \right| &\lesssim \int_{|x|\geq R} |v| |u|^2 dx \\
		&\lesssim \|v\|_{L^4} \|u\|_{L^4} \|u\|_{L^2(|x|\geq R)} \\
		&\lesssim \|\nabla v\|_{L^2} \|\nabla u\|_{L^2} \|u\|_{L^2(|x| \geq R)} \\
		&\lesssim \|u\|_{L^2(|x| \geq R)} + 2 \|v\|_{L^2(|x|\geq R)}.
		\end{align*}
		Collecting the above estimates, we get the following result.
		\begin{lemma} \label{lem-blow-non-rad-pro}
			Let $d=4$ and $\kappa=\frac{1}{2}$. Let $R>0$ and $\chi_R$ be as in \eqref{def-chi-R}. Then there exists $C>0$ independent of $R$ such that 
			\begin{align*}
			\frac{d^2}{dt^2} V_{\chi_R}(t) \leq 16 E(u,v) &+ CR^{-2} \left(\|u(t)\|^2_{L^2(|x|\geq R)} + 2 \|v(t)\|^2_{L^2(|x| \geq R)}\right) \\
			&+ C \left(\|u(t)\|^2_{L^2(|x|\geq R)} + 2 \|v(t)\|^2_{L^2(|x| \geq R)}\right).
			\end{align*}
		\end{lemma}
		We are now complete the proof of Theorem $\ref{theo-blo-non-rad}$. Applying Lemma $\ref{lem-con-L2-nor}$ and Lemma $\ref{lem-blow-non-rad-pro}$, we get for any $\epsilon>0$ and $R>0$, there exists $C>0$ independent of $\epsilon$ and $R$ such that for any $t\in [0,T_0]$ with $T_0:= \frac{\epsilon R}{C}$,
		\[
		\frac{d^2}{dt^2} V_{\chi_R}(t) \leq 16 E(u_0,v_0) + CR^{-2}(o_R(1)+\epsilon) + C(o_R(1)+\epsilon).
		\]
		Note that the constant $C$ may change from lines to lines but is independent of $\epsilon$ and $R$. We choose $\epsilon>0$ such that $C\epsilon = -8E(u_0,v_0)>0$. We next choose $R\gg 1$ large enough such that $CR^{-2}(o_R(1)+\epsilon) + Co_R(1)= - 4 E(u_0,v_0)$. Then
		\[
		\frac{d^2}{dt^2} V_{\chi_R}(t) \leq 4 E(u_0,v_0)
		\]
		for any $t\in [0,T_0]$. Integrating from $0$ to $T_0$, we get
		\begin{align*}
		V_{\chi_R}(T_0) &\leq V_{\chi_R}(0) + V'_{\chi_R}(0) T_0 + 2 E(u_0,v_0) T_0^2 \\
		&\leq V_{\chi_R}(0) + V'_{\chi_R} \frac{\epsilon R}{C} + 2 E(u_0,v_0) \frac{\epsilon^2 R^2}{C^2}.
		\end{align*}
		We next claim that
		\[
		V_{\chi_R}(0)=o_R(1) R^2, \quad V'_{\chi_R}(0)= o_R(1) R.
		\]
		Using this claim, we have
		\[
		V_{\chi_R}(T_0) \leq o_R(1) R^2 + 2 \eta R^2, 
		\]
		where $\eta:= \frac{E(u_0,v_0) \epsilon^2}{C^2}<0$. Taking $R\gg 1$ large enough, we obtain 
		\[
		V_{\chi_R}(T_0) \leq \eta R^2 <0,
		\]
		which contradicts to the fact $V_{\chi_R}(T_0)$ is non-negative. It remains to prove the claim. We seet that for $R\gg 1$ large enough,
		\begin{align*}
		V_{\chi_R}(0) &= \int \chi_R (|u_0|^2 + 2 |v_0|^2) dx \\
		&= \int_{|x|  \leq \sqrt{R}} |x|^2 (|u_0|^2 +2 |v_0|^2) dx + \int_{\sqrt{R}<|x| <2R} \chi_R(|u_0|^2 + 2 |v_0|^2)dx \\
		&\mathrel{\phantom{= \int_{|x|  \leq \sqrt{R}} |x|^2 (|u_0|^2 +2 |v_0|^2) dx}}+ \int_{|x|\geq 2R} 2 (|u_0|^2 + 2|v_0|^2) dx \\
		&\leq R M(u_0,v_0) + R^2 \int_{|x|>\sqrt{R}} |u_0|^2 + 2 |v_0|^2 dx \\
		&= o_R(1) R^2.
		\end{align*}
		We also have
		\begin{align*}
		V'_{\chi_R}(0) &= 2 \int \nabla \chi_R \cdot \text{Im}(\nabla u_0 \overline{u}_0 + \nabla v_0 \overline{v}_0) dx \\
		&= 2 \int \frac{\chi'_R}{r} x \cdot \text{Im} (\nabla u_0 \overline{u}_0 + \nabla v_0 \overline{v}_0) dx \\
		&=2 \int_{|x| \leq \sqrt{R}} \frac{\chi'_R}{r} x \cdot \text{Im} (\nabla u_0 \overline{u}_0 + \nabla v_0 \overline{v}_0) dx + 2 \int_{\sqrt{R} <|x| <2R} \frac{\chi'_R}{r} x \cdot \text{Im} (\nabla u_0 \overline{u}_0 + \nabla v_0 \overline{v}_0) dx \\
		&\leq 4\sqrt{R} \left(\|\nabla u_0\|_{L^2} \|u_0\|_{L^2}+ \|\nabla v_0\|_{L^2} \|v_0\|_{L^2}\right) \\
		&\mathrel{\phantom{\leq}}+ 8 R \left(\|\nabla u_0\|_{L^2(|x|>\sqrt{R})} \|u_0\|_{L^2(|x|>\sqrt{R})}+ \|\nabla v_0\|_{L^2(|x|>\sqrt{R})} \|v_0\|_{L^2(|x|>\sqrt{R})}\right) \\
		&= o_R(1) R. 
		\end{align*}
		The claim is proved and the proof of Theorem \ref{theo-blo-non-rad} is complete.
	\hfill $\Box$
	
	We end this subsection by giving the proof of Theorem $\ref{theo-blo-rad}$.
	
	\noindent \textit{Proof of Theorem $\ref{theo-blo-rad}$.}
		We use the localized virial estimate \eqref{loc-vir-est} to have
		\begin{align*}
		\frac{d^2}{dt^2} V_{\chi_R}(t) \leq 16 E(u(t),v(t)) &- 4 \int (\chi_{1,R} - C\epsilon \chi_{2,R}^2) |\nabla u(t)|^2 dx \nonumber \\
		& -2 \int (\chi_{1,R} - C\epsilon \chi_{2,R}^2) |\nabla v(t)|^2 dx + O(R^{-2} + \epsilon R^{-2} + \epsilon^{-1/3} R^{-2}),
		\end{align*}
		where
		\[
		\chi_{1,R}= 2-\chi''_R, \quad \chi_{2,R} = 8-\Delta \chi_R.
		\]
		If we choose a suitable function $\chi_R$ defined in \eqref{def-chi-R} so that
		\begin{align} \label{pos-con}
		\chi_{1,R}(r) - C\epsilon \chi_{2,R}(r)^2 \geq 0, \quad \forall r \geq 0,
		\end{align}
		for a suifficiently small $\epsilon>0$, then by choosing $R\gg 1$ large enough depending on $\epsilon$, we obtain
		\[
		\frac{d^2}{dt^2} V_{\chi_R}(t) \leq 8 E(u_0,v_0) <0,
		\]
		for any $t$ in the existence time. The classical argument of Glassey \cite{Glassey} implies that the solution must blows up in finite time. 
		
		To finish the proof, let us show \eqref{pos-con}. We use the argument of \cite{OT}. Let us define the following function
		\[
		\vartheta(r):= \left\{
		\begin{array}{c l c}
		2r & \text{if}& 0\leq r\leq 1, \\
		2[r-(r-1)^3] &\text{if}& 1<r\leq 1+1/\sqrt{3}, \\
		\vartheta' <0 &\text{if}& 1+ 1/\sqrt{3} <r < 2, \\
		0 &\text{if}& r\geq 2,
		\end{array}
		\right.
		\]
		and 
		\[
		\theta(r):= \int_0^r \vartheta(s)ds.
		\]
		We see that the function $\theta$ satisfies \eqref{def-the}. We next define $\chi_R$ as in \eqref{def-chi-R}. With this choice of $\chi_R$, the condition \eqref{pos-con} is satisfied. Indeed,
		for $0 \leq r \leq R$, \eqref{pos-con} is trivial since $\chi_{1,R} = \chi_{2,R}=0$. For $R<r <(1+1/\sqrt{3})R$, we have
		\[
		\chi_{1,R}= 2-\chi''_R= 2-\theta''(r/R) = 2-\vartheta'(r/R) = 6 (r/R-1)^2,
		\]
		and
		\begin{align*}
		\chi_{2,R}=8-\Delta \chi_R= 8-\chi''_R - \frac{3}{r} \chi'_R = 2-\chi''_R + 3(2-\chi'_R/r) &= 6(r/R-1)^2 \left(1+\frac{r/R-1}{r/R} \right) \\
		&< 6(r/R-1)^2 (1+1/\sqrt{3}).
		\end{align*}
		By choosing $\epsilon>0$ small enough, we see that the condition \eqref{pos-con} is satisfied. For $r>(1+1/\sqrt{3})R$, we see that $\chi''_R \leq 0$. Thus $\chi_{1,R}=2-\chi''_R \geq 2$ and $\chi_{2,R}=2-\chi''_R + 3(2-\chi'_R/r) \leq C$ for some constant $C>0$. Therefore, the condition \eqref{pos-con} is also satisfied by choosing $\epsilon>0$ small enough. The proof is complete.
	\hfill $\Box$
	
	\subsection{Characterization of finite time blow-up solutions with minimal mass}
	\label{Sub:2.2}
	The main purpose of this subsection is to classify finite time blow-up solutions of \eqref{mas-res-Syst} with minimal mass. We want to show that if $(u,v)$ is a solution which blows up at finite time $T$, then up to symmetries of the system, it is the pseudo-conformal transformation of the ground states $(\phi, \psi)$ for \eqref{sys-ell-equ} with $\omega_2 = 2 \omega_1 =2$.
		
	To this end, we need the profile decomposition related to \eqref{mas-res-Syst}. 
	\begin{proposition}[Profile decomposition] \label{prop-pro-dec}
		Let $d=4$. Le $(u_n,v_n)_{n\geq 1}$ be a bounded sequence in $H^1 \times H^1$. Then there exist a subsequence, still denoted by $(u_n,v_n)_{n\geq 1}$, a family $(x^j_n)_{n\geq 1}$ of sequences in $\mathbb{R}^4$ and a sequence $(U^j, V^j)_{j\geq 1}$ of $H^1\times H^1$-functions such that
		\begin{itemize}
			\item[(1)] for every $j\ne k$, 
			\begin{align} \label{ort-pro-dec}
			|x^j_n-x^k_n| \rightarrow \infty \text{ as } n\rightarrow \infty; 
			\end{align}
			\item[(2)] for every $l\geq 1$ and every $x \in \mathbb{R}^4$, 
			\[
			u_n(x) = \sum_{j=1}^l U^j(x-x^j_n) + u^l_n(x), \quad v_n(x)= \sum_{j=1}^l V^j(x-x^j_n) + v^l_n(x),
			\]
			with
			\begin{align} \label{err-pro-dec}
			\limsup_{n\rightarrow \infty} \|(u^l_n, v^l_n)\|_{L^q\times L^q} \rightarrow 0 \text{ as } l \rightarrow \infty,
			\end{align}
			for every $q\in (2, 4)$.
		\end{itemize}
		Moreover, for every $l\geq 1$,
		\begin{align}
		M(u_n,v_n) &= \sum_{j=1}^l M(U^j_n, V^j_n) + M(u^l_n,v^l_n) + o_n(1), \label{mas-pro-dec} \\
		K(u_n,v_n) &= \sum_{j=1}^l K(U^j,V^j) + K(u^l_n, v^l_n) + o_n(1), \label{kin-pro-dec} \\
		P(u_n,v_n) &= \sum_{j=1}^l P(U^j,V^j) + P(u^l_n, v^l_n) + o_n(1), \label{sca-pro-dec}
		\end{align}
		where $o_n(1) \rightarrow 0$ as $n\rightarrow \infty$.
	\end{proposition}
	
	\begin{proof}
		The proof is based on the argument of \cite{HK}. For reader's convenience, we recall some details. Let $(\vb{u,v})=(u_n,v_n)_{n\geq 1}$ be a bounded sequence in $H^1 \times H^1$. Since $H^1 \times H^1$ is a Hilbert space, we denote $\Omega(\vb{u,v})$ the set of functions obtained as weak limits of sequences of $(u_n(\cdot+x_n), v_n(\cdot+x_n))_{n\geq 1}$ with $(x_n)_{n\geq 1}$ a sequence in $\mathbb{R}^4$. Denote
		\[
		\eta(\vb{u,v}) := \sup \{ M(u,v) + K(u,v) \ : \ (u,v) \in \Omega(\vb{u,v})\}.
		\]
		If $\eta(\vb{u,v})=0$, then we can take $(U^j, V^j)=(0,0)$ for all $j\geq 1$. Otherwise we choose $(U^1,V^1) \in \Omega(\vb{u,v})$ such that
		\[
		M(U^1,V^1) + K(U^1, V^1) \geq \frac{1}{2} \eta (\vb{u,v})>0.
		\]
		By definition of $\Omega(\vb{u,v})$, there exists a sequence $(x^1_n)_{n\geq 1} \subset \mathbb{R}^4$ such that up to a subsequence,
		\[
		(u_n(\cdot+x^1_n), v_n(\cdot+x^1_n)) \rightharpoonup (U^1,V^1) \text{ weakly in } H^1 \times H^1.
		\]
		Set $u^1_n(x) := u_n(x) - U^1(x-x^1_n)$ and $v^1_n(x) := v_n(x) - V^1(x-x^1_n)$. It follows that $(u^1_n(\cdot+x^1_n), v^1_n(\cdot+x^1_n)) \rightharpoonup (0,0)$ weakly in $H^1 \times H^1$ and 
		\begin{align*}
		M(u_n,v_n) &= M(U^1,V^1) + M(u^1_n,v^1_n) + o_n(1), \\
		K(u_n,v_n) &= K(U^1, V^1) + K(u^1_n, v^1_n) + o_n(1).
		\end{align*}
		We next show that 
		\[
		P(u_n,v_n) = P(U^1, V^1) + P(u^1_n, v^1_n) + o_n(1).
		\]
		We have
		\begin{align*}
		\langle v_n, (u_n)^2 \rangle &= \langle V^1(\cdot-x^1_n) + v^1_n, (U^1(\cdot-x^1_n) + u^1_n)^2 \rangle \\
		& = \langle V^1(\cdot-x^1_n), [U^1(\cdot-x^1_n)]^2 \rangle + \langle v^1_n, (u^1_n)^2 \rangle + R_{1,n} \\
		& = \langle V^1, [U^1]^2 \rangle + \langle v^1_n, (u^1_n)^2 \rangle + R_{1,n},
		\end{align*}
		where 
		\begin{align*}
		R_{1,n}:= \langle V^1(\cdot-x^1_n), (u^1_n)^2 \rangle &+ \langle v^1_n, [U^1(\cdot-x^1_n)]^2 \rangle \\
		& + 2 \langle v^1_n, U^1(\cdot-x^1_n) u^1_n \rangle + 2 \langle V^1(\cdot-x^1_n), U^1(\cdot-x^1_n) u^1_n \rangle.
		\end{align*}
		Since $(u^1_n(\cdot+x^1_n), v^1_n(\cdot+x^1_n) ) \rightharpoonup (0,0)$ weakly in $H^1\times H^1$, we see that $R_{1,n}= o_n(1)$. We now replace $(\vb{u,v})=(u_n, v_n)_{n\geq 1}$ by $(\vb{u}^1,\vb{v}^1)=(u^1_n, v^1_n)_{n\geq 1}$ and repeat  the same process. If $\eta(\vb{u}^1,\vb{v}^1)=0$, then we choose $(U^j,V^j) = (0,0)$ for all $j\geq 2$. Otherwise there exists $(U^2,V^2)\in \Omega(\vb{u}^1,\vb{v}^1)$ and a sequence $(x^2_n)_{n\geq 1} \subset \mathbb{R}^4$ such that
		\[
		M(U^2,V^2) + K(U^2,V^2) \geq \frac{1}{2} \eta(\vb{u}^1,\vb{v}^1) >0,
		\]
		and $(u^1_n(\cdot+x^2_n), v^1_n(\cdot+x^2_n)) \rightharpoonup (U^2,V^2)$ weakly in $H^1 \times H^1$. Set $u^2_n(x) := u^1_n(x) - U^2(x-x^2_n)$ and $v^2_n(x):= v^1_n(x) - V^2(x-x^2_n)$. It follows that $(u^2_n(\cdot+x^2_n), v^2_n(\cdot+x^2_n)) \rightharpoonup (0,0)$ weakly in $H^1 \times H^1$. We also have
		\begin{align*}
		M(u^1_n,v^1_n) &= M(U^2,V^2) + M(u^2_n, v^2_n) + o_n(1), \\
		K(u^1_n, v^1_n) &= K(U^2, V^2) + K(u^2_n, v^2_n) + o_n(1), \\
		P(u^1_n, v^1_n) &= P(U^2, V^2) + P(u^2_n, v^2_n) + o_n(1).
		\end{align*}
		We next claim that $|x^1_n-x^2_n|\rightarrow \infty$ as $n\rightarrow \infty$. Indeed, if it is not true, then up to a subsequence, $x^1_n -x^2_n \rightarrow x_0$ as $n\rightarrow \infty$ for some $x_0 \in \mathbb{R}^4$. Since 
		\[
		u^1_n(x+ x^2_n) = u^1_n(x + (x^2_n-x^1_n) + x^1_n), \quad v^1_n(x+ x^2_n) = v^1_n( x+ (x^2_n-x^1_n) + x^1_n),
		\]
		and $(u^1_n( \cdot+ x^1_n), v^1_n(\cdot+x^1_n))\rightharpoonup (0,0)$, it yields that $(U^2,V^2) =(0,0)$ which is a contradiction to $\eta(\vb{u}^1,\vb{v}^1)>0$. 
		
		An argument of iteration and orthogonal extraction allow us to construct family $(x^j_n)_{j\geq 1}$ of sequences in $\mathbb{R}^4$ and a sequence $(U^j,V^j)_{j\geq 1}$ of $H^1 \times H^1$-functions satisfying the conclusion of Proposition $\ref{prop-pro-dec}$. To complete the proof, let us show \eqref{err-pro-dec}. Since the series $\sum_{j=1}^\infty M(U^j, V^j) + K(U^j, V^j)$ is convergent, it follows that
		\[
		M(U^j,V^j) + K(U^j,V^j) \rightarrow 0 \text{ as } j \rightarrow \infty. 
		\]
		By construction, we get $\eta(\vb{u}^j,\vb{v}^j) \rightarrow 0$ as $j \rightarrow \infty$. We now introduce $\theta: \mathbb{R}^4 \rightarrow [0,1]$ satisfying $\theta(\xi)=1$ for $|\xi| \leq 1$ and $\theta(\xi)=0$ for $|\xi|\geq 2$. For $R>0$, we define 
		\[
		\hat{\chi}_R(\xi):= \theta(\xi/R),
		\]
		where $\hat{\cdot}$ is the Fourier transform. We next write
		\[
		u^l_n= \chi_R \ast u^l_n + (\delta-\chi_R) \ast u^l_n, \quad v^l_n = \chi_R \ast v^l_n + (\delta-\chi_R) \ast v^l_n,
		\]
		where $\ast$ is the convolution operator and $\delta$ is the Dirac delta function. Let $q \in (2,4)$ be fixed. Using Sobolev embedding and the Plancherel formula, we have
		\begin{align*}
		\|(\delta-\chi_R) \ast (u^l_n, v^l_n)\|^2_{L^q \times L^q} &\lesssim \|(\delta-\chi_R) \ast (u^l_n,v^l_n)\|^2_{\dot{H}^\gamma \times \dot{H}^\gamma} \\
		&\lesssim \int |\xi|^{2\gamma} |(1-\hat{\chi}_R(\xi))|^2 \left(|\hat{u}^l_n(\xi)|^2 + |\hat{v}^l_n(\xi)|^2\right) d\xi \\
		&\lesssim R^{2\gamma-2} \|(u^l_n, v^l_n)\|^2_{\dot{H}^1\times \dot{H}^1},
		\end{align*}
		where $\gamma=2-\frac{2}{q} \in (0,1)$. On the other hand, by H\"older inequality, we have
		\begin{align*}
		\|\chi_R \ast (u^l_n,v^l_n)\|_{L^q \times L^q} &\lesssim \|\chi_R \ast (u^l_n, v^l_n)\|^{\frac{2}{q}}_{L^2 \times L^2} \|\chi_R \ast (u^l_n, v^l_n)\|^{1-\frac{2}{q}}_{L^\infty\times L^\infty} \\
		& \lesssim \|(u^l_n, v^l_n)\|^{\frac{2}{q}}_{L^2 \times L^2} \|\chi_R \ast (u^l_n, v^l_n)\|^{1-\frac{2}{q}}_{L^\infty \times L^\infty}.
		\end{align*}
		Using the fact
		\[
		\limsup_{n\rightarrow \infty} \|\chi_R \ast (u^l_n,v^l_n)\|_{L^\infty \times L^\infty} = \sup_{x_n} \limsup_{n\rightarrow \infty} |\chi_R \ast u^l_n(x_n)| + |\chi_R \ast v^l_n(x_n)|,
		\]
		the definition of $\Omega(\vb{u}^l,\vb{v}^l)$ implies that
		\begin{align*}
		\limsup_{n\rightarrow \infty} &\|\chi_R \ast (u^l_n,v^l_n)\|_{L^\infty \times L^\infty} \\
		&\leq \sup \left\{\left| \int \chi_R(-x) u(x) dx\right| + \left| \int \chi_R(-x) v(x) dx \right| \ : \ (u,v) \in \Omega(\vb{u}^l,\vb{v}^l) \right\}.
		\end{align*}
		By the Plancherel formula, we get
		\begin{align*}
		\left|\chi_R(-x) u(x) dx\right| + \left| \int \chi_R (-x) v(x) dx \right| & = \left| \int \hat{\chi}_R(\xi) \hat{u}(\xi) d\xi \right| + \left| \int \hat{\chi}_R(\xi) \hat{v}(\xi) d\xi \right| \\
		&\lesssim \|\hat{\chi}_R\|_{L^2} \|(u,v)\|_{L^2\times L^2} \\
		&\lesssim R^2 \eta(\vb{u}^l,\vb{v}^l).
		\end{align*}
		We obtain for every $l\geq 1$,
		\begin{align*}
		\|(u^l_n,v^l_n)\|_{L^q \times L^q} &\lesssim R^{\gamma-1} \|(u^l_n, v^l_n)\|_{\dot{H}^1 \times \dot{H}^1} +  \left[ R^2 \eta(\vb{u}^l,\vb{v}^l) \right]^{1-\frac{2}{q}} \|(u^l_n,v^l_n)\|^{\frac{2}{q}}_{L^2 \times L^2}.
		\end{align*}
		We now choose $R= \left[\eta(\vb{u}^l,\vb{v}^l)\right]^{\frac{1}{2}-\epsilon}$ for some $\epsilon>0$ small enough, we get
		\[
		\|(u^l_n, v^l_n)\|_{L^q \times L^q} \lesssim [\eta(\vb{u}^l,\vb{v}^l)]^{(1-\gamma)\left(\frac{1}{2}-\epsilon\right)} \|(u^l_n, v^l_n)\|_{\dot{H}^1 \times \dot{H}^1}  + [\eta(\vb{u}^l,\vb{v}^l)]^{2\epsilon \left(1-\frac{2}{q}\right)} \|(u^l_n, v^l_n)\|_{L^2\times L^2}^{\frac{2}{q}}.
		\]
		Since $\eta(\vb{u}^l,\vb{v}^l) \rightarrow 0$ as $l\rightarrow \infty$, the uniform boundedness in $H^1 \times H^1$ of $(u^l_n, v^l_n)_{l\geq 1}$ implies that
		\[
		\limsup_{n\rightarrow \infty} \|(u^l_n, v^l_n)\|_{L^q \times L^q} \rightarrow 0 \text{ as } l \rightarrow \infty.
		\]
		The proof is complete.
	\end{proof}
	
	Using the profile decomposition, we have the following compactness lemma. 
	\begin{lemma} \label{lem-com-lem}
		Let $d=4$ and $\kappa=\frac{1}{2}$. Let $(u_n,v_n)_{n\geq 1}$ be a bounded sequence in $H^1 \times H^1$ satisfying
		\[
		\limsup_{n\rightarrow \infty} K(u_n,v_n) \leq A, \quad \limsup_{n\rightarrow \infty} P(u_n,v_n) \geq a.
		\]
		Then there exists $(x_n)_{n\geq 1} \subset \mathbb{R}^d$ such that up to a subsequence 
		\[
		(u_n(\cdot+x_n), v_n(\cdot + x_n)) \rightharpoonup (U,V) \text{ weakly in } H^1 \times H^1
		\]
		for some $(U,V) \in H^1 \times H^1$ satisfying 
		\[
		M(U,V) \geq \frac{4a^2}{A^2} M(\phi_0,\psi_0),
		\]
		where $(\phi_0, \psi_0)$ is as in Theorem $\ref{theo-sha-GN-cla}$.
	\end{lemma}
	\begin{proof}
		Since $(u_n,v_n)_{n\geq 1}$ is bounded in $H^1\times H^1$, we can apply the profile decomposition given in Proposition \ref{prop-pro-dec}. By \eqref{sca-pro-dec}, we have
		\[
		a \leq \limsup_{n\rightarrow \infty} P(u_n,v_n) \leq \limsup_{n\rightarrow \infty} \left(\sum_{j=1}^l P(U^j,V^j) + P(u^l_n,v^l_n) \right) \leq \sum_{j=1}^\infty P(U^j,V^j).
		\]
		Here we use the fact that $\limsup_{n\rightarrow \infty} P(u^l_n,v^l_n) \rightarrow 0$ as $l\rightarrow \infty$ which follows easily from H\"older's inequality and \eqref{err-pro-dec}. Using the sharp Gagliardo-Nirenberg inequality, we bound
		\[
		a \leq \sum_{j=1}^\infty P(U^j,V^j) \leq \frac{\sup_{j\geq 1} \sqrt{M(U^j,V^j)}}{2\sqrt{M(\phi_0,\psi_0)}} \sum_{j=1}^\infty K(U^j,V^j).
		\]
		We have from \eqref{kin-pro-dec} that 
		\[
		\sum_{j=1}^\infty K(U^j,V^j) \leq \limsup_{n\rightarrow \infty} K(u_n,v_n) \leq A.
		\]
		It follows that
		\[
		\sup_{j\geq 1} M(U^j,V^j) \geq \frac{4a^2}{A^2} M(\phi_0, \psi_0).
		\]
		Since $\sum_{j=1}^\infty M(U^j,V^j)$ is convergent, the above supremum is attained. There thus exists $j_0$ such that $M(U^{j_0}, V^{j_0}) \geq \frac{4a^2}{A^2}M(\phi_0, \psi_0)$. We also have
		\[
		u_n(x+x^{j_0}_n) = U^{j_0}(x) + \sum_{1 \leq j \leq l \atop j \ne j_0} U^j(x+ x^{j_0}_n - x^j_n) + u^l_n(x+x^{j_0}_n),
		\]
		and similarly for $v_n(x+x^{j_0}_n)$. Note that the pairwise orthogonality \eqref{ort-pro-dec} implies that for $j \ne j_0$,
		\[
		(U^j(\cdot + x^{j_0}-x^j_n), V^j(\cdot+x^{j_0}_n - x^j_n)) \rightharpoonup (0,0) \text{ weakly in } H^1 \times H^1.
		\]
		We thus get
		\[
		(u_n(\cdot + x^{j_0}_n), v_n(\cdot + x^{j_0})) \rightharpoonup (U^{j_0}, V^{j_0}) + (\tilde{u}^l, \tilde{v}^l) \text{ weakly in } H^1 \times H^1,
		\]
		where $(\tilde{u}^l, \tilde{v}^l)$ is the weak limit in $H^1 \times H^1$ of $(u^l_n(\cdot+x^{j_0}_n), v^l_n(\cdot + x^{j_0}_n))$. Moreover, by \eqref{err-pro-dec},
		\[
		\|(\tilde{u}^l,\tilde{v}^l)\|_{L^q \times L^q} \leq \limsup_{n\rightarrow \infty} \|(u^l_n(\cdot + x^{j_0}_n), v^l_n(\cdot + x^{j_0}_n))\|_{L^q \times L^q} = \limsup_{n\rightarrow \infty} \|(u^l_n,v^l_n)\|_{L^q \times L^q} \rightarrow 0
		\]
		as $l \rightarrow \infty$. By the uniqueness of the weak limit, we get $(\tilde{u}^l, \tilde{v}^l) = (0,0)$ for every $l\geq j_0$. Thus
		\[
		(u_n(\cdot+ x^{j_0}_n), v_n(\cdot+ x^{j_0}_n)) \rightharpoonup (U^{j_0},V^{j_0}) \text{ weakly in } H^1 \times H^1.
		\]
		The proof is complete.
	\end{proof}
		
	To classify blow-up solutions with minimal mass, we also need the following lemma.
	\begin{lemma} \label{lem-cha-str}
		Let $d=4$ and $\kappa=\frac{1}{2}$. Let $(u,v) \in H^1 \times H^1$ be such that
		\[
		M(u,v) = M_{\emph{gs}}, \quad E(u,v) = 0.
		\]
		Then there exist $(\phi,\psi) \in \mathcal{G}$, $\theta_1, \theta_2 \in \mathbb{R}$ and $\rho>0$ such that 
		\[
		(u(x), v(x)) = (e^{i\theta_1} \rho^2 \phi(\rho x), e^{i\theta_2} \rho^2 \psi(\rho x)).
		\]
	\end{lemma}
	\begin{proof}
		Since $M(u,v) = M_{\text{gs}}$ and $E(u,v)=0$, we see that 
		\[
		J(u,v) = \frac{K(u,v) \sqrt{M(u,v)}}{P(u,v)} = 2 \sqrt{M_{\text{gs}}}.
		\]
		This implies that $(u,v)$ is a minimizer of $J$. On the other hand, since $\|\nabla |u|\|_{L^2} \leq \|\nabla u\|_{L^2}$ and $P(|u|,|v|) \geq P(u,v)$, it follows that $J(|u|,|v|) \leq J(u,v)$ or $(|u|,|v|)$ is also a minimizer of $J$. Note that any positive minimizer of $J$ is radial. In fact, suppose that there exists $(u_0,v_0)$ a positive minimizer of $J$ which is not radial. Let $(u^*_0,v^*_0)$ be the symmetric rearrangement of $(u_0,v_0)$. We have from \cite[Theorem 3.4]{LL} that
		\[
		\int v_0 (u_0)^2 dx < \int v^*_0 (u^*_0)^2 dx,
		\]
		and also $\|\nabla u^*_0\|_{L^2} \leq \|\nabla u_0\|_{L^2}, \|u^*_0\|_{L^2} = \|u_0\|_{L^2}$. We get $J(u^*_0, v^*_0) < J(u_0,v_0)$ which is a contradiction. Therefore, $(|u|,|v|)$ is radial. We learn from the proof of \cite[Theorem 5.1]{HOT} that there exists $(\phi,\psi) \in \mathcal{G}$ and $\rho>0$ such that
		\[
		|u(x)| = \rho^a \phi(\rho x), \quad |v(x)| = \rho^a \psi(\rho x)
		\]
		for some $a \in \mathbb{R}$. Since $M(|u|,|v|) = M(u,v) = M_{\text{gs}}= M(\phi,\psi)$, it follows that $a=2$. It remains to show that $\tilde{u}(x)= \frac{u(x)}{|u(x)|}$ and $\tilde{v}(x) = \frac{v(x)}{|v(x)|}$ are constant on $\mathbb{R}^4$. This fact follows by the same lines as in the proof of Item (4) of Theorem $\ref{theo-exi-min}$. The proof is complete.
	\end{proof}
	
	We also need the following Cauchy-Schwarz inequality due to Banica \cite{Banica}
	\begin{lemma} \label{lem-cau-sch-ine}
		Let $d=4$ and $\kappa=\frac{1}{2}$. Let $(u,v) \in H^1\times H^1$ be such that $M(u,v) = M_{\emph{gs}}$. It follows that for any real-valued function $\varphi \in C^1$ satisfying $\nabla \varphi$ is bounded, 
		\[
		\left| \int \nabla \varphi \cdot \emph{Im} (\nabla u \overline{u} + \nabla v \overline{v} ) dx \right| \leq \sqrt{2 E(u,v)} \left(\int |\nabla \varphi|^2 (|u|^2 + 2 |v|^2 ) dx\right)^{1/2}.
		\]
	\end{lemma}
	\begin{proof}
		We first note that if $M(u,v)=M_{\text{gs}}$, then $E(u,v) \geq 0$. This fact follows easily from the sharp Gagliardo-Nirenberg inequality. Since $M(e^{is \varphi} u, e^{2is\varphi} v)= M(u,v) = M_{\text{gs}}$, we get $E(e^{is \varphi} u, e^{2is\varphi} v) \geq 0$ for any real number $s$. On the other hand,
		\begin{align*}
		E(e^{is \varphi} u, e^{2is\varphi} v) &= \frac{1}{2} K(e^{is \varphi} u, e^{2is\varphi} v) - P(e^{is \varphi} u, e^{2is\varphi} v) \\
		&= \frac{1}{2} \left( \int |is u\nabla \varphi + \nabla u|^2 dx + \frac{1}{2} \int |2is v \nabla \varphi + \nabla v|^2 dx \right) - P(u,v) \\
		&= \frac{s^2}{2} \int |\nabla \varphi|^2 (|u|^2 +2 |v|^2 ) dx + s \int \nabla \varphi \cdot \text{Im }( \nabla u \overline{u} +  \nabla v \overline{v} ) dx + E(u,v).
		\end{align*}
		Since $E(e^{is \varphi} u, e^{2is\varphi} v) \geq 0$ for any $s \in \mathbb{R}$, the discriminant of the equation in $s$ must be non-positive, and the result follows.
	\end{proof}
	
	We also need the virial identity related to \eqref{mas-res-Syst}.
	\begin{lemma} \label{lem-vir-ide}
		Let $d=4$ and $\kappa=\frac{1}{2}$. Let $(u_0,v_0) \in H^1 \times H^1$ be such that $(|x| u_0, |x| v_0) \in L^2 \times L^2$. Then the corresponding solution to \eqref{mas-res-Syst} satisfies
		\begin{align} \label{vir-ide-1}
		\frac{d^2}{dt^2} \left( \|x u(t)\|^2_{L^2} + 2 \|x v(t)\|^2_{L^2} \right) = 16 E(u_0,v_0).
		\end{align}
		In particular,
		\begin{align} \label{vir-ide-2}
		\|x u(t)\|^2_{L^2} + 2 \|x v(t)\|^2_{L^2} &= \|x u_0\|^2_{L^2} + 2 \|x v_0\|^2_{L^2} + 4 t \emph{Im} \int x \cdot \nabla u_0 \overline{u}_0 + x \cdot \nabla v_0 \overline{v}_0 dx + 8t^2 E(u_0,v_0) \nonumber \\
		&= 8t^2 E\left(e^{i\frac{|x|^2}{4t}} u_0, e^{i\frac{|x|^2}{2t}} v_0 \right).
		\end{align}
	\end{lemma}
	\begin{proof}
		The identity \eqref{vir-ide-1} follows from Lemma $\ref{lem-vir-ide-chi}$ with $\chi(x)= |x|^2$. The identity \eqref{vir-ide-2} follows from the fact that
		\begin{align*}
		\left| \nabla \left( e^{i\frac{|x|^2}{4t}} u_0 \right) \right|^2 &= \frac{1}{4t^2} |xu_0|^2 + \frac{1}{t} \text{Im} (x \cdot \nabla u_0 \overline{u}_0) + |\nabla u_0|^2, \\
		\left| \nabla \left( e^{i\frac{|x|^2}{2t}} v_0 \right) \right|^2 &= \frac{1}{t^2} |xv_0|^2 + \frac{2}{t} \text{Im} (x \cdot \nabla v_0 \overline{v}_0) + |\nabla v_0|^2.
		\end{align*}
		The proof is complete.
	\end{proof}
	
	We are now able to show the characterization of finite time blow-up solutions of \eqref{mas-res-Syst} with minimal mass.
	
	\noindent \textit{Proof of Theorem $\ref{theo-cla-min-mas}$.}
		Let $(t_n)_{n\geq 1}$ be a time sequence satisfying $t_n \uparrow T$ as $n\rightarrow \infty$. Set 
		\[
		\varrho^2_n:= \frac{K(\phi_0,\psi_0)}{K(u(t_n),v(t_n))}, \quad (u_n(x), v_n(x)) := (\varrho_n^2 u(t_n, \varrho_n x), \varrho^2_n v(t_n, \varrho_n x)),
		\]
		where $(\phi_0, \psi_0)$ is as in Theorem $\ref{theo-sha-GN-cla}$. 
		
		We first claim that there exist $(\tilde{x}_n)_{n\geq 1} \subset \mathbb{R}^4$ such that
		\[
		M(u(t_n), v(t_n)) dx -M(\phi_0, \psi_0) \delta_{x=\tilde{x}_n} \rightharpoonup 0 \text{ as } n\rightarrow \infty,
		\]
		where $\delta_{x=\tilde{x}_n}$ is the Dirac measure at $x=\tilde{x}_n$. 		Indeed, since $\|(u(t_n), v(t_n))\|_{H^1 \times H^1} \rightarrow \infty$ as $t_n \uparrow T$, the conservation of mass implies that $K(u(t_n), v(t_n)) \rightarrow \infty$ as $n\rightarrow \infty$. Thus $\varrho_n \rightarrow 0$ as $n\rightarrow \infty$. Moreover,
		\begin{align*}
		M(u_n, v_n) &= M(u(t_n), v(t_n)) = M(u_0,v_0) = M_{\text{gs}}, \\
		K(u_n, v_n) &= \varrho^2_n K(u(t_n), v(t_n)) = K(\phi_0, \psi_0), \\
		E(u_n, v_n) &= \varrho^2_n E(u(t_n), v(t_n)) = \varrho^2_n E(u_0, v_0) \rightarrow 0 \text{ as } n \rightarrow \infty.
		\end{align*}
		In particular, the last convergence implies that $P(u_n, v_n) \rightarrow \frac{1}{2} K(\phi_0, \psi_0)$ as $n\rightarrow \infty$. The sequence $(u_n, v_n)_{n\geq 1}$ satisfies conditions of Lemma $\ref{lem-com-lem}$ with $A= K(\phi_0,\psi_0)$ and $a= \frac{1}{2} K(\phi_0, \psi_0)$. There thus exist $(x_n)_{n\geq 1} \subset \mathbb{R}^d$ and $(U,V) \in H^1 \times H^1$ such that
		\[
		(u_n(\cdot+ x_n), v_n(\cdot+x_n)) \rightharpoonup (U,V) \text{ weakly in } H^1 \times H^1 \text{ as } n\rightarrow \infty
		\]
		with
		\[
		M(U,V) \geq \frac{4a^2}{A^2}  M(\phi_0,\psi_0) = M(\phi_0,\psi_0)=M_{\text{gs}}.
		\]
		By the semi-continuity of weak convergence, we have
		\[
		M(U,V) \leq \liminf_{n\rightarrow \infty} M(u_n(\cdot+x_n), v_n(\cdot+x_n)) = \liminf_{n\rightarrow \infty} M(u_n,v_n) = M_{\text{gs}}.
		\]
		Thus $M(U,V)=\lim_{n\rightarrow \infty} M(u_n(\cdot+x_n), v_n(\cdot+x_n)) = M_{\text{gs}}$. Therefore, $(u_n(\cdot+x_n), v_n(\cdot+x_n)) \rightarrow (U,V)$ strongly in $L^2 \times L^2$ as $n\rightarrow \infty$. On the other hand,
		we estimate
		\begin{align*}
		|P(u_n,v_n)-P(U,V)| &\leq \left|\text{Re} \int (\overline{v}_n - V) (u_n)^2 dx \right| + \left|\text{Re} \int  \overline{V}(u_n - U) (u_n +U)dx \right| \\
		&\leq \|v_n-V\|_{L^2} \|u_n\|^2_{L^4} + \|V\|_{L^4} \|u_n-U\|_{L^2} \|u_n+U\|_{L^4}.
		\end{align*}
		Using the Sobolev embedding $L^4 \subset H^1$ and the fact $(u_n, v_n)_{n\geq 1}, (U,V)$ are bounded in $H^1 \times H^1$, the strong $L^2\times L^2$-convergence $(u_n(\cdot+x_n), v_n(\cdot+x_n)) \rightarrow (U,V)$ implies that $P(u_n,v_n)\rightarrow P(U,V)$ as $n\rightarrow \infty$. We thus get from the sharp Gagliardo-Nirenberg inequality that
		\[
		\frac{1}{2} K(\phi_0,\psi_0) = \lim_{n\rightarrow \infty} P(u_n,v_n) = P(U,V) \leq \frac{1}{2} \sqrt{\frac{M(U,V)}{M_{\text{gs}}}} K(U,V) = \frac{1}{2} K(U,V).
		\]
		Besides, the semi-continuity of weak convergence implies that
		\[
		K(U,V) \leq \liminf_{n\rightarrow \infty} K(u_n(\cdot+x_n), v_n(\cdot+x_n)) = \liminf_{n\rightarrow \infty} K(u_n,v_n)= K(\phi_0,\psi_0).
		\]
		Thus $K(U,V)= \lim_{n\rightarrow \infty} K(u_n(\cdot+x_n), v_n(\cdot+x_n)) = K(\phi_0,\psi_0)$. Therefore $(u_n(\cdot+x_n), v_n(\cdot+x_n)) \rightarrow (U,V)$ strongly in $\dot{H}^1 \times \dot{H}^1$. Hence $(u_n(\cdot+x_n), v_n(\cdot+x_n)) \rightarrow (U,V)$ strongly in $H^1 \times H^1$ as $n\rightarrow \infty$. Moreover, $M(U,V)= M_{\text{gs}}, K(U,V)=K(\phi_0, \psi_0)$ and $P(U,V)=\frac{1}{2} K(\phi_0,\psi_0)$. There thus exists $(U,V) \in H^1 \times H^1$ satisfying $M(U,V)= M_{\text{gs}}$ and $E(U,V)=0$. Lemma $\ref{lem-cha-str}$ implies that there exist $(\tilde{\phi},\tilde{\psi})\in \mathcal{G}$, $\vartheta_1, \vartheta_2 \in \mathbb{R}$ and $\varrho >0$ such that
		\[
		(U(x), V(x)) = (e^{i\vartheta_1} \varrho^2 \tilde{\phi}(\varrho x),e^{i\vartheta_2} \varrho^2 \tilde{\psi}(\varrho x) ).
		\]
		Thus
		\[
		(u_n(\cdot+x_n), v_n(\cdot+x_n)) = (\varrho^2_n u(t_n, \varrho_n(\cdot+x_n)), \varrho^2_n v(t_n, \varrho_n(\cdot+x_n))) \rightarrow (e^{i\vartheta_1} \varrho^2 \tilde{\phi}(\varrho \cdot), e^{i\vartheta_2} \varrho^2 \tilde{\psi}(\varrho \cdot)) 
		\]
		strongly in $H^1 \times H^1$ as $n\rightarrow \infty$. Set 
		\[
		\tilde{\varrho}_n := \varrho_n/\varrho, \quad \tilde{x}_n := \varrho_n x_n, \quad \tilde{\vartheta}_1 = -\vartheta_1, \quad \tilde{\vartheta}_2 = - \vartheta_2.
		\]
		We obtain 
		\begin{align} \label{mea-con-pro}
		(e^{i\tilde{\vartheta}_1} \tilde{\varrho}^2_n u(t_n, \tilde{\varrho}_n \cdot + \tilde{x}_n), e^{i\tilde{\vartheta}_2} \tilde{\varrho}^2_n v(t_n, \tilde{\varrho}_n \cdot + \tilde{x}_n) ) \rightarrow (\tilde{\phi},\tilde{\psi})
		\end{align}
		strongly in $H^1 \times H^1$ as $n\rightarrow \infty$. This yields the claim. To see this, we change variable $x= \tilde{\varrho}_n y + \tilde{x}_n$ to have for any $\varphi \in C^\infty_0$ that
		\begin{align*}
		\int &(|u(t_n,x)|^2 + 2|v(t_n,x)|^2) \varphi(x) dx \\
		&= \int \tilde{\varrho}^4_n (|u(t_n,\tilde{\varrho}_n y + \tilde{x}_n)|^2 + 2|v(t_n,\tilde{\varrho}_n y + \tilde{x}_n)|^2) \varphi(\tilde{\varrho}_n y + \tilde{x}_n) dx \\
		&= \int \left[ \left( |\tilde{\varrho}_n^2 u(t_n, \tilde{\varrho}_n y + \tilde{x}_n)|^2 + 2|\tilde{\varrho}^2_n v(t_n, \tilde{\varrho}_n y + \tilde{x}_n)|^2  \right) - (|\tilde{\phi}(y)|^2 + 2 |\tilde{\psi}(y)|^2 ) \right] \varphi(\tilde{\varrho}_n y + \tilde{x}_n) dx \\
		&\mathrel{\phantom{=}} + \int (|\tilde{\phi}(y)|^2 + 2 |\tilde{\psi}(y)|^2) ( \varphi(\tilde{\varrho}_n y + \tilde{x}_n) - \varphi(\tilde{x}_n) ) dy + \int (|\tilde{\phi}(y)|^2 + 2 |\tilde{\psi}(y)|^2) \varphi(\tilde{x}_n) dy.
		\end{align*}
		This implies that
		\begin{align*}
		\left| \int \right. & \left. (|u(t_n,x)|^2 + 2|v(t_n,x)|^2) \varphi(x) dx - M(\tilde{\phi},\tilde{\psi}) \varphi(\tilde{x}_n) \right| \\
		&\leq \|\varphi\|_{L^\infty} \int \left(|\tilde{\varrho}_n^2 u(t_n, \tilde{\varrho}_n y + \tilde{x}_n)|^2 - |\tilde{\phi}(y)|^2 \right) + 2 \left(|\tilde{\varrho}_n^2 v(t_n, \tilde{\varrho}_n y + \tilde{x}_n)|^2 - |\tilde{\psi}(y)|^2 \right) dx \\
		&\mathrel{\phantom{\leq}} + \int (|\tilde{\phi}(y)|^2 + 2 |\psi(y)|^2) |\varphi(\tilde{\varrho}_n y + \tilde{x}_n) - \varphi(\tilde{x}_n) | dy.
		\end{align*}
		By \eqref{mea-con-pro}, we have $|\tilde{\varrho}^2_n u(t_n, \tilde{\varrho}_n \cdot + \tilde{x}_n)|^2 \rightarrow |\tilde{\phi}|^2$ in $L^1$ and $|\tilde{\varrho}^2_n v(t_n, \tilde{\varrho}_n \cdot + \tilde{x}_n)|^2 \rightarrow |\tilde{\psi}|^2$ in $L^1$ as $n\rightarrow \infty$, this implies the first integral in the right hand side vanishes as $n\rightarrow \infty$. Moreover, since $\tilde{\varrho}_n \rightarrow 0$, the second integral in the right hand side also vanishes by the dominated convergence. The claim is thus proved.
		
		We now able to show the classification of finite time blow-up solutions with minimal mass. Up to subsequence, we may assume that $\tilde{x}_n \rightarrow x_0 \in \{0, \infty\}$. Now let $\chi$ be a smooth non-negative radial compactly supported function satisfying $\chi(x) = |x|^2$ if $|x| <1$ and $|\nabla \chi(x)|^2 \leq C \chi(x)$ for some constant $C>0$. For $R>0$, we define
		\[
		\chi_R(x) = R^2 \chi(x/R), \quad I_R(t):= \int \chi_R(x) (|u(t,x)|^2 + 2 |v(t,x)|^2) dx.
		\]
		Using the Cauchy-Schwarz inequality given in Lemma $\ref{lem-cau-sch-ine}$, we get
		\begin{align*}
		|I'_R(t)| &= 2 \left| \int \nabla \chi_R \cdot \text{Im} (\nabla u \overline{u} + \nabla v \overline{v}) dx \right| \\
		&\leq 2 \sqrt{2 E(u_0,v_0)} \left( \int |\nabla \chi_R|^2 ( |u(t)|^2 + 2 |v(t)|^2) dx \right)^{1/2} \\
		& \leq 2 C \sqrt{E(u_0,v_0)} \left( \int \chi_R (|u(t)|^2 + 2 |v(t)|^2) dx 	\right)^{1/2} \\
		&= C(u_0,v_0) \sqrt{I_R(t)}.
		\end{align*}
		Integrating with respect to $t$, we get
		\begin{align} \label{cla-pro-1}
		|\sqrt{I_R(t)} - \sqrt{I_R(t_n)}| \leq C(u_0,v_0) |t_n-t|.
		\end{align}
		It follows from the claim above that $I_R(t_n) \rightarrow 0$ as $n\rightarrow \infty$. Indeed, if $|x_n| \rightarrow 0$, then $I_R(t_n) \rightarrow M(\phi,\psi) \chi_R(0)=0$ as $n\rightarrow \infty$. If $|x_n| \rightarrow \infty$, then $I_R(t_n) \rightarrow 0$ since $\chi_R$ is compactly supported. Taking $n\rightarrow \infty$ in \eqref{cla-pro-1}, we obtain
		\[
		I_R(t) \leq C(u_0,v_0) (T-t)^2. 
		\]
		Taking $R \rightarrow \infty$, we get
		\[
		8t^2 E\left(e^{i\frac{|x|^2}{4t}} u_0, e^{i\frac{|x|^2}{2t}} v_0 \right) = \|xu(t)\|^2+ 2 \|x v(t)\|^2_{L^2} \leq C(u_0,v_0)(T-t)^2. 
		\]
		By letting $t \rightarrow T$, we see that $E\left(e^{i\frac{|x|^2}{4T}} u_0, e^{i\frac{|x|^2}{2T}} v_0 \right)=0$. Moreover, 
		\[
		M\left(e^{i\frac{|x|^2}{4T}} u_0, e^{i\frac{|x|^2}{2T}} v_0 \right) = M(u_0,v_0) = M_{\text{gs}}.
		\]
		Lemma $\ref{lem-cha-str}$ then implies that there exist $(\phi,\psi) \in \mathcal{G}$, $\tilde{\theta}_1, \tilde{\theta}_2 \in \mathbb{R}$ and $\tilde{\rho}>0$ such that
		\[
		e^{i\frac{|x|^2}{4T}} u_0(x) = e^{i\tilde{\theta}_1} \tilde{\rho}^2 \phi(\tilde{\rho} x), \quad e^{i\frac{|x|^2}{2T}} v_0 (x) = e^{i\tilde{\theta}_2} \tilde{\rho}^2 \psi(\tilde{\rho} x). 
		\]
		Redefining $\tilde{\rho} = \frac{\rho}{T}$ and $\tilde{\theta}_1 = \theta_1 + \frac{\rho^2}{T}$ and $\tilde{\theta}_2 = \theta_2 + \frac{\rho^2}{T}$, we obtain
		\[
		u_0(x) = e^{i\theta_1} e^{i\frac{\rho^2}{T}} e^{-i\frac{|x|^2}{4T}} \left( \frac{\rho}{T}\right)^2 \phi\left(\frac{\rho x}{T}\right), \quad v_0(x) = e^{i\theta_2} e^{i\frac{\rho^2}{T}} e^{-i\frac{|x|^2}{2T}} \left( \frac{\rho}{T}\right)^2 \psi\left(\frac{\rho x}{T}\right).
		\]
		By the uniqueness of solution to \eqref{mas-res-Syst}, it follows that $(u(t), v(t))$ is given as in \eqref{sol-cla}. The proof is complete.
	\hfill $\Box$

	\section*{Acknowledgement}
	The author would like to express his deep gratitude to his wife - Uyen Cong for her encouragement and support. He also would like to thank the reviewer for his/her helpful comments and suggestions.


\begin{thebibliography}{10}
		
		\bibitem{AB}
		J.~Albert and S.~Bhattarai.
		\newblock Existence and stability of a two-parameter family of solitary waves for an NLS-KdV system.
		\newblock {\em Adv. Differential Equations}, 18:1129--1164, 2013.
		
		\bibitem{Ardila}
		A.~H.~Ardila.
		\newblock Orbital stability of standing waves for a system of nonlinear Schr\"odinger equations with three wave interaction.
		\newblock {\em Nonlinear Anal.}, 167:1--20, 2018.
		
		\bibitem{Banica}
		V.~Banica.
		\newblock Remarks on the blow-up for the Schr\"odinger equation with critical mass on a plane domain.
		\newblock {\em Ann. Sc. Norm. Super. Pisa Cl. Sci.}, 3:(5)139--170, 2004.
		
		\bibitem{Bhattarai-syst}
		S.~Bhattarai.
		\newblock Stability of normalized solitary waves for three coupled nonlinear Schr\"odinger equations.
		\newblock {\em Discrete Contin. Dyn. Syst-A}, 36:1789--1811, 2016.
		
		\bibitem{Bhattarai-frac}
		S.~Bhattarai.
		\newblock Existence and stability of standing waves for nonlinear Schr\"odinger systems involving the fractional Laplacian.
		\newblock preprint, \url{arXiv:1604.01718}, 2016.
		
		\bibitem{Brock}
		F.~Brock.
		\newblock A general rearrangement inequality \`a la Hardy-Littlewood.
		\newblock {\em J. Inequality Appl.}, 5:309--320, 2000.
				
		\bibitem{CCO}
		M.~Colin, Th.~Colin and M.~Ohta.
		\newblock Stability of solitary waves for a system of nonlinear Schr\"odinger equations with three wave interaction.
		\newblock {\em Ann. Inst. H. Poincar\'e Anal. Non Lin\'eaire}, 26(6):2211-2226, 2009.
		
		\bibitem{CG}
		E.~Csobo and F.~Genoud.
		\newblock Minimal mass blow-up solutions for the $L^2$ critical NLS with inverse-square potential.
		\newblock {\em Nonlinear Anal.}, 168:110--129, 2018.
				
		\bibitem{Dinh}
		V.~D.~Dinh.
		\newblock Blowup of $H^1$ solutions for a class of the focusing inhomogeneous nonlinear Schr\"odinger equation.
		\newblock {\em Nonlinear Anal.}, 174:169--188, 2018.
		
		\bibitem{DWZ}
		D.~Du, Y.~Wu and K.~Zhang.
		\newblock On blow-up criterion for the nonlinear Schr\"odinger equation.
		\newblock {\em Discrete Contin. Dyn. Syst.}, 36(7):3639--3650, 2016.
		
		\bibitem{Glassey}
		R.~T.~Glassey.
		\newblock On the blowing up of solutions to the Cauchy problem for nonlinear Schr\"odinger equations.
		\newblock {\em J. Math. Phys.}, 18:1794--1797, 1977.
		
		\bibitem{Hamano}
		M. Hamano.
		\newblock Global dynamics below the ground state for the quadratic Schr\"odinger system in 5D.
		\newblock preprint, arXiv: \url{arXiv:1805.12245}, 2018.
		
		\bibitem{HOT}
		N. Hayashi, T. Ozawa and K. Tanaka.
		\newblock On a system of nonlinear Schr\"odinger equations with quadratic interaction.
		\newblock {\em Ann. Inst. H. Poincar\'e Anal. Non Lin\'eaire}, 30:661--690, 2013.
		
		\bibitem{HK}
		T.~Hmidi and S. Keraani.
		\newblock Blwoup theory for the critical nonlinear Schr\"odinger equations revisited.
		\newblock {\em Int. Math. Res. Not.}, 46:2815--2828, 2005.
		
		\bibitem{LL}
		E.~H.~Lieb and M.~Loss.
		\newblock {\em Analysis}.
		\newblock American Mathematical Society, 2000.
		
		\bibitem{Lions1}
		P.~L. Lions.
		\newblock The concentration-compactness method in the calculus of variations. The locally compact case I. 
		\newblock { \em Ann. Inst. H. Poincar\'e Anal. Non Lin\'eaire}, 1:109--145, 1984.
		
		\bibitem{Lions2}
		P.~L. Lions.
		\newblock The concentration-compactness method in the calculus of variations. The locally compact case II. 
		\newblock {\em Ann. Inst. H. Poincar\'e Anal. Non Lin\'eaire}, 1:223--283, 1984.
		
		\bibitem{Strauss}
		W.~A.~Strauss.
		\newblock Existence of solitary waves in higher dimensions.
		\newblock {\em Comm. Math. Phys.}, 55(2):149--162, 1977.
		
		\bibitem{OT}
		T.~Ogawa and Y. Tsutsumi.
		\newblock Blow-up of $H^1$ solutions for the nonlinear Schr\"odinger equation.
		\newblock {\em J. Differential Equations}, 92:317--330, 1990.
		
		\bibitem{Tao}
		T.~Tao.
		\newblock {\em Nonlinear Dispersive Equations: Local and Global Analysis}, CBMS 106.
		\newblock American Mathematical Society, Providence, 2006.
				
	\end{thebibliography}
\end{document}